\documentclass[12pt]{amsart}
\usepackage{amsmath,amssymb,pst-all,graphicx}
\usepackage{amscd, amsthm}
\usepackage[cmtip,arrow]{xy}
\usepackage{pb-diagram, pb-xy}

\newcommand{\CC}{\mathbb C}

\newcommand{\PP}{\mathbb P}

\newcommand{\ZZ}{\mathbb Z}

\newcommand{\mcA}{\mathcal{A}}
\newcommand{\mcB}{\mathcal B}
\newcommand{\mcC}{\mathcal C}

\newcommand{\mcF}{\mathcal F}
\newcommand{\mcG}{\mathcal {G}}
\newcommand{\mcL}{\mathcal L}
\newcommand{\mcO}{\mathop {\mathcal O}\nolimits}

\newcommand{\mcS}{\mathcal{S}}
\newcommand{\mcT}{\mathcal T}
\newcommand{\mcW}{\mathcal{W}}

\DeclareMathOperator{\HH}{H}

\newcommand{\gtS}{{\mathfrak S}}

\newcommand{\Sing}{\mathop {\rm Sing}\nolimits}
\newcommand{\GCD}{\mathop {\rm GCD}\nolimits}

\newcommand{\Aut}{\mathop {\rm {Aut}}\nolimits}

\newcommand{\id}{\mathop {{\rm id}}\nolimits}

\newcommand{\parti}{\mathfrak{p}}
\newcommand{\Parti}{\mathfrak{P}}
\newcommand{\sm}{\mathrm{sm}}
\newcommand{\pr}{\mathrm{pr}}
\newcommand{\Irr}{\mathrm{Irr}}
\newcommand{\LB}{\mathrm{LB}}
\newcommand{\Stab}{\mathrm{Stab}}

\newcommand{\NV}{\mathrm{NV}}

\newcommand{\I}{\mathop {\rm I}\nolimits}

\newtheorem{thm}{Theorem}[section] % 1st argument is your name for it
\newtheorem{lem}[thm]{Lemma}     % 2nd argument is what is printed
\newtheorem{cor}[thm]{Corollary}
\newtheorem{prop}[thm]{Proposition}

\newtheorem{claim}{Claim}
\theoremstyle{definition}
\newtheorem{defin}[thm]{Definition}
\newtheorem{rem}[thm]{Remark}
\newtheorem{ex}[thm]{Example}
\newtheorem{prob}[thm]{Problem}
\begin{document}

\title[Galois covers of graphs and embedded topology]{Galois covers of graphs and \\ embedded topology of plane curves}
\author{Taketo Shirane}
\address{Department of Mathematical Sciences, Faculty of Science and Technology, Tokushima University, Tokushima, 770-8502, Japan}
\email{shirane@tokushima-u.ac.jp}
%\date{
%%\today
%}
\maketitle

\begin{abstract}
The splitting number is effective to distinguish the embedded topology of plane curves, and it is not determined by the fundamental group of the complement of the plane curve. 
In this paper, we give a generalization of the splitting number, called the \textit{splitting graph}. 
By using the splitting graph, we classify the embedded topology of plane curves consisting of one smooth curve and non-concurrent three lines, called \textit{Artal arrangements}. 
\end{abstract}

\section{Introduction}

In this paper, we study the embedded topology of plane curves in the complex projective plane $\PP^2:=\CC\PP^2$, 
such as in knot and link theory. 
Here the \textit{embedded topology} of a plane curve $\mcC\subset\PP^2$ is the homeomorphic class of the pair $(\PP^2,\mcC)$ of $\PP^2$ and the reduced divisor $\mcC$ on $\PP^2$. 
We introduce a new invariant, called the \textit{splitting graph}, using a Galois cover of graphs, which describe the ``splitting" of a plane curve by a Galois cover of $\PP^2$. 

The first result about the embedded topology of plane curves is given by O.~Zariski \cite{zariski}. 
He studied the existence of a certain algebraic function for a given plane curve, 
and pointed out that the existence can be reduced to finding the fundamental group of the complement of the given curve (the word \textit{complement} is often omitted for short). 
He also proved that the fundamental group of a plane sextic curve with six cusps is isomorphic to the free product $\ZZ_2\ast\ZZ_3$ if the six cusps lies on a conic, and the fundamental group is not isomorphic to $\ZZ_2\ast\ZZ_3$ otherwise (in fact, it is isomorphic to $\ZZ_6$ proved by M.~Oka \cite{okaz}), where $\ZZ_i$ is the cyclic group of order $i$. 
This result shows that the configuration of singularities of a plane curve can affect the fundamental group (hence the embedded topology) of the plane curves.

It is known that, if two plane curves have the same embedded topology, then they have the same combinatorics, i.e., they are equisingular (see \cite{act} for details). 
However Zariski's result shows that the converse is false. 
A pair of plane curves with the same combinatorics is called \textit{Zariski pair} if they have different embedded topology. 
From the 90's, E.~Artal \cite{artal}, M.~Oka \cite{oka}, H.~Tokunaga \cite{tokunaga, tokunaga-dihedral}, and others have discovered Zariski pairs by studying fundamental groups of plane curves. 
In 21st century, $\pi_1$-equivalent Zariski pairs have been discovered (cf. \cite{degtyarev, shirane}), 
and the following problem has arisen naturally. 
(Here two plane curves are said to be \textit{$\pi_1$-equivalent} if their fundamental groups are isomorphic.)

\begin{prob}
Give a method for distinguishing the embedded topology of $\pi_1$-equivalent plane curves. 
\end{prob}

Several methods succeed in distinguishing the embedded topology of certain $\pi_1$-equivalent plane curves. 
For example, there are methods using the theory of K3-surfaces \cite{degtyarev}, the braid monodromy \cite{acc}, the splitting number \cite{shirane} and the linking set \cite{benoit-meilhan}. 
The methods using the theory of K3-surfaces and the splitting numbers are based on techniques of algebraic geometry. 
On the other hand, the ones using the braid monodromy and the linking set are derived from invariants of geometrical topology. 
In \cite{benoit-shirane}, B.~Guerville and the author gave a relation between the splitting number and the linking set for certain plane curves. 

In this paper, we give a generalization of the splitting number, called \textit{splitting graph}. 
The splitting number is based on the studies of splitting curves for double covers by Artal--Tokunaga \cite{artal-tokunaga}, Tokunaga \cite{tokunaga-splitting} and S.~Bannai \cite{bannai}. 
Here a plane curve $\mcC\subset\PP^2$ is said to be \textit{splitting} for a double cover $\phi:X\to\PP^2$ if $\phi^{\ast}(\mcC)=\mcC^++\mcC^-$ for some two curves $\mcC^+,\mcC^-\subset X$ with no common components and $\phi(\mcC^\pm)=\mcC$. 
In \cite{artal-tokunaga}, Artal and Tokunaga implicitly used splitting curves for double covers to establish the difference of the fundamental groups of plane curves. 
In \cite{tokunaga-splitting}, Tokunaga defined the splitting curves with respect to double covers, and study the splitting curves as an analog of elementary number theory (see \cite[Remark~0.1]{tokunaga-splitting}). 
In \cite{bannai}, Bannai introduced the \textit{splitting type} of certain plane curves for double covers. 
The splitting type gives a method for distinguishing the embedded topology of plane curves without going through the fundamental groups. 
In \cite{shirane}, the author defined the \textit{splitting number} of irreducible curves for Galois covers, 
and proved that the splitting number is invariant under certain homeomorphisms from $\PP^2$ to itself based on Bannai's idea. 
Moreover, by using the splitting number,
he distinguished the embedded topology of the $\pi_1$-equivalent equisingular curves defined by Shimada \cite{shimada}. 
In \cite{shirane2}, the \textit{connected number} of plane curves (possibly reducible) for Galois covers of $\PP^2$ was defined, 
which is a modification of the splitting number. 
It is known that the splitting number and the connected number are not determined by the fundamental group (see \cite{shirane, shirane2}). 
These results show that studying the ``splitting" of plane curves for Galois covers is effective to distinguish the embedded topology. 
In this paper, we define the splitting graph of plane curves for Galois covers to describe more precisely how a plane curve splits by a Galois cover, which is not determined by the fundamental group (see Remark~\ref{rem: fundamental group} (ii)). 

As an application of the splitting graph, we classify the embedded topology of Artal arrangements, where an \textit{Artal arrangement} is a plane curve consisting of one smooth curve and non-concurrent three lines. 
The name ``Artal" comes from E.~Artal who gave the first Zariski pair of Artal arrangements in \cite{artal}. 
Note that Artal arrangements have been defined in \cite{ban-ben-shi-tok} and \cite{shirane2}, and 
our definition is a generalization of the ones in \cite{ban-ben-shi-tok} and \cite{shirane2}. 
An Artal arrangement in \cite{ban-ben-shi-tok} or \cite{shirane2} is a plane curve consisting of one smooth curve and its non-concurrent three total inflectional tangents, i.e., each of the three tangents intersects at just one point with the smooth curve. 
In this paper, we define Artal arrangements of type $\Parti$ for a triple $\Parti$ of partitions of an integer $d\geq3$. 
In our definition, the assumption ``total inflectional tangents" is removed 
(see Section~\ref{sec: artal arrangement} for details). 
Moreover, we define the splitting graph associated to an Artal arrangement. 
By Theorem~\ref{thm: artal arrangement}, we obtain the following theorem. 

\begin{thm}\label{thm: introduction} 
Let $\mcA_1$ and $\mcA_2$ be two Artal arrangements of type $\Parti$ for a triple $\Parti$ of partitions of an integer $d\geq3$. 
Then $\mcA_1$ and $\mcA_2$ have the same embedded topology if and only if the splitting graph associated to $\mcA_1$ is equivalent to the splitting graph associated to $\mcA_2$. 
\end{thm}

This paper is organized as follows. 
In Section~\ref{sec: graph}, we investigate branched Galois covers of graphs in order to define and distinguish the splitting graph. 
In particular, we define the \textit{net voltage class} of a closed walk on a graph for a Galois cover over the graph (Definition~\ref{def: gap class}), and give a method of distinguishing two Galois covers of a graph (Corollary~\ref{cor: gap class}). 
In Section~\ref{sec: splitting graph}, we define the \textit{splitting graph} of a plane curve for a Galois cover over $\PP^2$ as a Galois cover of certain graphs (Definition~\ref{def: splitting graph}). 
Moreover, we introduce a method of computing net voltage classes for a cyclic cover (Theorem~\ref{thm: gap class}). 
In the final section, we define \textit{Artal arrangements} of type $(\parti_1,\parti_2,\parti_3)$ for three partitions $\parti_i$ of an integer $d\geq 3$ (Definition~\ref{def: artal arrangement}), and classify the embedded topology of Artal arrangements by using the splitting graphs (Theorem~\ref{thm: artal arrangement}). 

\section{Galois covers of graphs}\label{sec: graph}

In this section, we consider `branched covers' of graphs and their equivalence. 
Unramified covers of graphs has been investigated by Gross--Tucker \cite{gross-tucker}, Stark--Terras \cite{terras-stark1,terras-stark2, terras-stark3} and so on. 
In this section, we investigate how to distinguish branched Galois covers over a graph.

In this paper, we assume that any \textit{graph} is finite, 
i.e., the sets of vertices and edges are finite. 
Note that we allow a graph to be disconnected. 
The sets of vertices and edges of a graph $\mcG$ are denoted by $V_\mcG$ and $E_\mcG$, respectively.
We consider each edge $e\in E_\mcG$ to have two directions, arbitrarily distinguished as the \textit{plus direction} $e^+$ and the \textit{minus derection} $e^-$, 
i.e., $e^\pm$ are two onto functions $\{0,1\}\to V_\mcG(e)$ such that $e^-(0)=e^+(1)$ and $e^-(1)=e^+(0)$, 
where $V_\mcG(e)$ is the endpoint set of the edge $e$. 
We call $e^\pm(0)$ and $e^\pm(1)$ the \textit{initial vertex} and the \textit{terminal vertex} of $e^\pm$, respectively. 
For two graphs $\mcG_i$ ($i=1,2$), a \textit{map} $\theta$ from $\mcG_1$ to $\mcG_2$ is a pair $\theta=(\theta_V,\theta_E)$ of maps $\theta_V:V_{\mcG_1}\to V_{\mcG_2}$ and $\theta_E:E_{\mcG_1}\to E_{\mcG_2}$ satisfying $V_{\mcG_2}(\theta_E(e))=\theta_V(V_{\mcG_1}(e))\subset V_{\mcG_2}$ for any $e\in E_{\mcG_1}$. 
A map $\theta:\mcG_1\to\mcG_2$ of graphs is called an \textit{isomorphism} if $\theta_V$ and $\theta_E$ are bijective.

\begin{defin}\label{def: Galois cover}
Let $\mcG$ be a graph, and let $G$ be a finite group. 
A \textit{Galois cover} of $\mcG$ with the Galois group $G$ (called \textit{$G$-cover} for short) is a map $\phi:\widetilde{\mcG}\to \mcG$ of graphs satisfying the following conditions; 
\begin{enumerate}
\item $G$ acts on $\widetilde{\mcG}$; 
\item the quotient map $q:\widetilde{\mcG}\to\widetilde{\mcG}/G$ corresponds to $\phi:\widetilde{\mcG}\to\mcG$, i.e., there exists an isomorphism $i:\widetilde{\mcG}/G\to\mcG$ such that $\phi=i\circ q$;
\item either $E_{\mcG}=\emptyset$ or $G$ acts on $E_{\widetilde{\mcG}}$ freely, i.e., for any $\tilde{e}\in E_{\widetilde{\mcG}}$, $g{\cdot}\tilde{e}=\tilde{e}$ if and only if $g=\id_G$. 
\end{enumerate}
A $G$-cover $\phi:\widetilde{\mcG}\to\mcG$ is said to be \textit{unbranched} or \textit{unramified} if the cardinality of $\phi_V^{-1}(v)$ is equal to the order $|G|$ of $G$ for each $v\in V_\mcG$, and \textit{branched} or \textit{ramified} otherwise. 
\end{defin}

\begin{rem}\label{rem: galois cover}
Let $\phi:\widetilde{\mcG}\to\mcG$ be a $G$-cover. 
\begin{enumerate}
\item
If $\phi$ is unramified, then it is called a \textit{regular covering} in \cite{gross-tucker}. 
\item
By Definition~\ref{def: Galois cover} (i), we have $V_{\widetilde{\mcG}}(g{\cdot}\tilde{e})=g{\cdot}V_{\widetilde{\mcG}}(\tilde{e})\subset V_{\widetilde{\mcG}}$. 
\item 
For each $v\in V_{\mcG}$ and $e\in E_\mcG$, $G$ acts transitively on the fibers $\phi_V^{-1}(v)$ and $\phi_E^{-1}(e)$, respectively, by Definition~\ref{def: Galois cover} (ii). 
\item
We assume that $\phi$ preserves the directions of edges, 
i.e., $\phi_V(\tilde{e}^+(0))=e^+(0)$ and $\phi_V(\tilde{e}^+(1))=e^+(1)$ for each $e\in E_{\mcG}$ and $\tilde{e}\in\phi_{E}^{-1}(e)$. 
Hence the action of $G$ on $\widetilde{\mcG}$ also preserves the direction of edges. 
\item %By condition (iii) in Definition~\ref{def: Galois cover}, 
The action of $G$ on $\widetilde{\mcG}$ is faithful if $E_{\mcG}\ne\emptyset$. 
\end{enumerate}
\end{rem}

We define the equivalence of Galois covers of graphs with a fixed Galois group $G$. 

\begin{defin}\label{def: equivalence of Galois covers}
Let $\mcG_i$ ($i=1,2$) be two graphs, and let $G$ be a finite group with an automorphism $\tau:G\to G$. 
Assume that there is an isomorphism $\theta:\mcG_1\to\mcG_2$. 
Let $\phi_i=(\phi_{V,i},\phi_{E,i}):\widetilde{\mcG}_i\to\mcG_i$ ($i=1,2$) be two $G$-covers. 
We say that $\widetilde{\mcG}_1$ and $\widetilde{\mcG}_2$ are \textit{$(\theta,\tau)$-equivalent}, written by $\widetilde{\mcG}_1\sim_{(\theta,\tau)}\widetilde{\mcG}_2$, if there is an isomorphism $\tilde{\theta}=(\tilde{\theta}_V,\tilde{\theta}_E):\widetilde{\mcG}_1\to\widetilde{\mcG}_2$ satisfying the following conditions; 
\begin{enumerate}
\item $\phi_2\circ\tilde{\theta}=\theta\circ\phi_1$, i.e., $\phi_{V,2}\circ\tilde{\theta}_V=\theta_V\circ\phi_{V,1}$ and $\phi_{E,2}\circ\tilde{\theta}_E=\theta_E\circ\phi_{E,1}$; 
\item for any $\tilde{v}\in V_{\widetilde{\mcG}_1}$ and $g\in G$, $\tilde{\theta}_V(g{\cdot}\tilde{v})=\tau(g){\cdot}\tilde{\theta}_V(\tilde{v})$; and
\item for any $\tilde{e}\in E_{\mcG_1}$ and $g\in G$, $\tilde{\theta}_E(g{\cdot}\tilde{e})=\tau(g){\cdot}\tilde{\theta}_{E}(\tilde{e})$. 
\end{enumerate}
\end{defin}

\begin{rem}
Let $\phi_i:\widetilde{\mcG}_i\to\mcG_i$, $\tau:G\to G$ and $\theta:\mcG_1\to\mcG_2$ be as in Definition~\ref{def: equivalence of Galois covers}. 
For $\tilde{v}_i\in V_{\widetilde{\mcG}_i}$, we put $\Stab_G(\tilde{v}_i):=\{g\in G\mid g{\cdot}\tilde{v}_i=\tilde{v}_i\}$.  
Assume that $E_{\mcG_i}=\emptyset$ for each $i=1,2$. 
Then $\widetilde{\mcG_1}\sim_{(\theta,\tau)}\widetilde{\mcG_2}$ if and only if, 
for each $v_1\in V_{\mcG_1}$ and $v_2:=\theta_V(v_1)$, 
$\tau(\Stab_G(\tilde{v}_1))$ is conjugate to $\Stab_G(\tilde{v}_2)$ for any $\tilde{v}_1\in\phi_{V,1}^{-1}(v_1)$ and $\tilde{v}_2\in\phi_{V,2}^{-1}(v_2)$. 
Indeed, since $E_{\widetilde{\mcG}_i}=\emptyset$, $\widetilde{\mcG}_1\sim_{(\theta,\tau)}\widetilde{\mcG}_2$ 
if and only if $\phi_{V,1}^{-1}(v_1)\sim_{(\theta,\tau)}\phi_{V,2}^{-1}(v_2)$ for each $v_1\in V_{\mcG_1}$, 
which is equivalent to the latter condition by regarding $\phi_{V,i}^{-1}(v_i)$ as $G$-sets. 
\end{rem}

In order to give a criterion for equivalence of Galois covers of graphs with edges, we introduce an invariant of Galois covers by using closed walks. 
A \textit{walk} $\gamma$ on a graph $\mcG$ is an alternating sequence of vertices and directed edges
\[ \gamma=(v_0,e_1^{\sigma_1},v_1,e_2^{\sigma_2},\dots,v_{n-1},e_n^{\sigma_n},v_n) \ \ \ (\sigma_i=+ \mbox{ or }-) \]
such that $e_i^{\sigma_i}(0)=v_{i-1}$ and $e_i^{\sigma_i}(1)=v_i$, i.e., the initial and terminal vertices of $e_i^{\sigma_i}$ are $v_{i-1}$ and $v_i$, respectively, for each $i=1,\dots,n$. 
For a walk $\gamma=(v_0,e_1^{\sigma_1},\dots,v_{n})$, we call $v_0$ and $v_{n}$ the \textit{initial vertex} and the \textit{terminal vertex} of $\gamma$, respectively, 
and we call the number $n$ of edges in $\gamma$ the \textit{length} of $\gamma$. 
A walk $\gamma$ is said to be \textit{closed} if the initial and terminal vertices coincide. 
For a Galois cover $\phi:\widetilde{\mcG}\to\mcG$ and a walk $\gamma=(v_0,e_1^{\sigma_1}\dots,v_{n})$ on $\mcG$, 
a \textit{lift} of $\gamma$ under $\phi$ is a walk $\tilde{\gamma}=(\tilde{v}_0,\tilde{e}_1^{\sigma_1},\dots,\tilde{v}_{n})$ on $\widetilde{\mcG}$ satisfying $\phi_V(\tilde{v}_i)=v_i$ and $\phi_E(\tilde{e}_i)=e_i$ for any $i=0,\dots,n$. 

Let $\phi:\widetilde{\mcG}\to\mcG$ be a $G$-cover of a graph $\mcG$, 
and let $\gamma=(v_0,e_1^{\sigma_1},\dots,v_{n})$ be a closed walk on $\mcG$. 
We fix a vertex $\tilde{v}_0$ such that $\phi_V(\tilde{v}_0)=v_0$. 
We define a subset $\widetilde{V}_{\phi,i}^{\gamma}(\tilde{v}_0)$ of $\phi_V^{-1}(v_i)$ for each $i=0,\dots,n$ as follows: 
\begin{align*} 
\widetilde{V}_{\phi,0}^{\gamma}(\tilde{v}_0)&:=\{ \tilde{v}_0 \}, \\
\widetilde{V}_{\phi,i}^{\gamma}(\tilde{v}_0)&:=\left\{ \tilde{v}_i\in\phi_V^{-1}(v_i) \,\middle|\, 
\begin{array}{l}
\tilde{e}_i^{\sigma_i}(0)\in\widetilde{V}_{\phi,i-1}^\gamma(\tilde{v}_0) \mbox{ and } \tilde{e}_i^{\sigma_i}(1)=\tilde{v}_i \\\mbox{for some $\tilde{e}_{i}\in\phi_E^{-1}(e_i)$} 
\end{array}
\right\}.
\end{align*}
In other words, $\widetilde{V}_{\phi,i}^{\gamma}(\tilde{v}_0)$ is the set of vertices $\tilde{v}_i\in \phi_V^{-1}(v_i)$ such that there is a lift of the walk $(v_0,e_i^{\sigma_1},\dots,v_i)$ under $\phi$ whose initial and terminal vertices are $\tilde{v}_0$ and $\tilde{v}_i$, respectively. 
Let $\NV_{\phi}(\gamma,\tilde{v}_0)$ be the following subset of $G$: 
\[ \NV_{\phi}(\gamma,\tilde{v}_0):=\{ g\in G \mid g{\cdot}\tilde{v}_0\in\widetilde{V}_{\phi,n}^{\gamma}(\tilde{v}_0) \}. \]
Note that $\widetilde{V}_{\phi,n}^{\gamma}(\tilde{v}_0)=\{g{\cdot}\tilde{v}_0\mid g\in \NV_{\phi}(\gamma,\tilde{v}_0)\}$ since the action of $G$ on $\phi_V^{-1}(v_0)$ is transitive by condition (ii) in Definition~\ref{def: Galois cover}. 

\begin{lem}\label{lem: conjugate}
Let $\phi:\widetilde{\mcG}\to\mcG$, $G$ and $\gamma$ be as above. 
For any $\tilde{v}_0\in\phi_V^{-1}(v_0)$ and $g\in G$, 
\[ \NV_{\phi}(\gamma, g{\cdot}\tilde{v}_0)=g \NV_{\phi}(\gamma, \tilde{v}_0) g^{-1}. \]
\end{lem}
\begin{proof}
For any $g_0\in G$, $g_0$ is an element of $\NV_{\phi}(\gamma,\tilde{v}_0)$ if and only if 
there exist a lift $\tilde{\gamma}=(\tilde{v}_0,\tilde{e}_1^{\sigma_1},\dots,\tilde{v}_{n})$ of $\gamma$ such that $\tilde{v}_{n}=g_0{\cdot}\tilde{v}_0$. 
By Remark~\ref{rem: galois cover} (iii), the latter condition is equivalent to the existence of a lift $\tilde{\gamma}'=(\tilde{v}_0',(\tilde{e}'_1)^{\sigma_1},\dots,\tilde{v}_{n}')$ of $\gamma$ such that $\tilde{v}_0'=g{\cdot}\tilde{v}_0$ and $\tilde{v}_{n}'=g{\cdot}\tilde{v}_n=gg_0g^{-1}{\cdot}\tilde{v}_0'$. 
Thus, $g_0\in \NV_{\phi}(\gamma,\tilde{v}_0)$ is equivalent to $gg_0g^{-1}\in \NV_{\phi}(\gamma,g{\cdot}\tilde{v}_0)$. 
Therefore, we obtain $\NV_{\phi}(\gamma,g{\cdot}\tilde{v}_0)=g\NV_{\phi}(\gamma,\tilde{v}_0)g^{-1}$. 
\end{proof}

For a subset $S$ of a group $G$, we call the family of sets 
\[ \{gSg^{-1}\subset G \mid g\in G\} \subset 2^G \]
the \textit{conjugacy class} of $S$ in $G$. 
By Lemma~\ref{lem: conjugate} and Remark~\ref{rem: galois cover} (ii), 
the conjugacy class of the set $\NV_{\phi}(\gamma, \tilde{v}_0)$ in $G$ does not depend on the choice of $\tilde{v}_0\in\phi_V^{-1}(v_0)$. 

\begin{defin}\label{def: gap class}
Let $\phi:\widetilde{\mcG}\to\mcG$, $G$ and $\gamma$ be as above. 
Let $\NV_{\phi}(\gamma)$ denote the conjugacy class of $\NV_{\phi}(\gamma,\tilde{v}_0)$, i.e.,
\[ \NV_{\phi}(\gamma):=\{g\NV_{\phi}(\gamma,\tilde{v}_0)g^{-1} \mid g\in G\}, \]
and we call $\NV_\phi(\gamma)$ the \textit{net voltage class} of $\gamma$ for $\phi$. 
\end{defin}

\begin{rem}
The name ``net voltage class" is derived from the net voltage on a walk in the base space of a voltage graph. 
Voltage graphs correspond to unramified covers of a graph, and the net voltage represents the terminal vertex of a lift of the walk (see \cite{gross-tucker}). 
In the case of unramified covers, a lift of a walk is uniquely determined by its initial vertex. 
However a lift of a walk is not uniquely determined by its initial vertex in the case of branched covers, 
hence a net voltage is not uniquely determined. 
\end{rem}

For a walk $\gamma=(v_0,e_{1}^{\sigma_1},\dots,v_{n})$ on a graph $\mcG_1$ and a map $\theta:\mcG_1\to \mcG_2$,  let $\theta(\gamma)$ denote the following walk on $\mcG_2$: 
\[  \theta(\gamma):=(\theta_V(v_0),\theta_E(e_1)^{\sigma_1'},\dots,\theta_V(v_{n})), \] 
where $\sigma_i'$ is the sign $\pm$ such that $\theta_E(e_i)^{\sigma_i'}(0)=\theta_V(v_{i-1})$ and $\theta_E(e_i)^{\sigma_i'}(1)=\theta_V(v_i)$. 
The net voltage class is invariant under equivalence of Galois covers of graphs by the following proposition.

\begin{prop}\label{prop: gap class}
Let $G$ be a finite group, let $\phi_i:\widetilde{\mcG}_i\to\mcG_i$ be a $G$-cover of graphs for each $i=1,2$, and let $\gamma_i$ be a closed walk on $\mcG_i$. 
Assume that $\widetilde{\mcG}_1\sim_{(\theta,\tau)}\widetilde{\mcG}_2$, where 
$\theta:\mcG_1\to\mcG_2$ and $\tau:G\to G$ are an isomorphism of graphs and an automorphism of $G$, respectively. 
If $\theta(\gamma_1)=\gamma_2$, then $\tau(\NV_{\phi_1}(\gamma_1))=\NV_{\phi_2}(\gamma_2)$. 
\end{prop}
\begin{proof}
Let $\tilde{\theta}:\widetilde{\mcG}_1\to\widetilde{\mcG}_2$ be an isomorphism satisfying conditions (i), (ii) and (iii) in Definition~\ref{def: equivalence of Galois covers}. 
By condition (i) in Definition~\ref{def: equivalence of Galois covers}, $\tilde{\theta}_V$ gives a bijection from $\phi_{V,1}^{-1}(v)$ to $\phi_{V,2}^{-1}(\theta_V(v))$ for each $v\in V_{\mcG_1}$. 
Let $v_0$ and $w_0$ be the initial vertices of $\gamma_1$ and $\gamma_2$, respectively, 
and let $n$ be the length of $\gamma_1$ and $\gamma_2$. 
Let $\tilde{v}_0$ be a vertex in $\phi_1^{-1}(v_0)$, and put $\tilde{w}_0:=\tilde{\theta}_V(\tilde{v}_0)$. 
Since $\tilde{\theta}$ is an isomorphism, 
an sequence $(\tilde{v}_0,\tilde{e}_1^{\sigma_1},\dots,\tilde{v}_{n})$ with $\tilde{v}_i\in V_{\widetilde{\mcG}_1}$ and $\tilde{e}_i\in E_{\widetilde{\mcG}_1}$ is a walk on $\widetilde{\mcG}_1$ if and only if 
$(\tilde{w}_0,(\tilde{e}_1')^{\sigma_1'},\dots,\tilde{w}_{n})$ is a walk on $\widetilde{\mcG}_2$ for certain signs $\sigma_i'$, 
where $\tilde{w}_i=\tilde{\theta}_V(\tilde{v}_i)$ and $\tilde{e}'_i:=\tilde{\theta}_E(\tilde{e}_i)$. 
Hence we have $\tilde{\theta}_V(\widetilde{V}_{\phi_1,n}^{\gamma_1}(\tilde{v}_0))=\widetilde{V}_{\phi_2,n}^{\gamma_2}(\tilde{w}_0)$. 
Moreover, we obtain $\tau(\NV_{\phi_1}(\gamma_1,\tilde{v}_0))=\NV_{\phi_2}(\gamma_2,\tilde{w}_0)$ by Definition~\ref{def: equivalence of Galois covers} (ii). 
Therefore, $\tau(\NV_{\phi_1}(\gamma_1))=\NV_{\phi_2}(\gamma_2)$ since $\tau$ is an automorphism of $G$. 
\end{proof}

A closed walk $\gamma=(v_0,e_1^{\sigma_1},\dots,v_{n})$ is said to be \textit{simple} if $v_i\ne v_j$ for any $0\leq i<j< n$. 
By the same idea of \cite{ban-ben-shi-tok}, 
we obtain the following corollary. 

\begin{cor}\label{cor: gap class}
Let $G$ be a finite group, let $\phi_i:\widetilde{\mcG}_i\to\mcG_i$ ($i=1,2$) be two $G$-covers of graphs, and let $\tau:G\to G$ be an automorphism. 
Let $\mcW_{\mcG_i}$ be the set of simple closed walks. 
If there exists no bijection $\Theta_{\mcW}:\mcW_{\mcG_1}\to\mcW_{\mcG_2}$ such that $\tau(\NV_{\phi_1}(\gamma_1))=\NV_{\phi_2}(\theta_\mcW(\gamma_1))$ for any $\gamma_1\in\mcW_{\mcG_1}$,
then there exist no isomorphisms $\theta:\mcG_1\to\mcG_2$ such that $\widetilde{\mcG}_1\sim_{(\theta,\tau)}\widetilde{\mcG}_2$. 
\end{cor}

\begin{proof}
If $\theta:\mcG_1\to\mcG_2$ is an isomorphism such that $\widetilde{\mcG}_1\sim_{(\theta,\tau)}\widetilde{\mcG}_2$, then $\theta$ induces a bijection $\Theta_\mcW:\mcW_{\mcG_1}\to\mcW_{\mcG_2}$ such that $\tau(\NV_{\phi_1}(\gamma_1))=\NV_{\phi_2}(\theta_\mcW(\gamma_2))$ by Proposition~\ref{prop: gap class}. 
\end{proof}

We investigate properties of the net voltage classes. 
For a walk $\gamma=(v_0,e_1^{\sigma_1},\dots,v_{n})$ with $v_{n}=v_0$, we call the closed walk 
\[\gamma^{-1}:=(v_{n},e_{n}^{-\sigma_n},\dots,e_1^{-\sigma_1},v_0)\] 
the \textit{inverse walk} of $\gamma$, 
where $-\sigma_i:=\mp$ if $\sigma_i=\pm$, respectively. 

\begin{lem}\label{lem: inverse walk}
Let $\phi:\widetilde{\mcG}\to\mcG$ be a $G$-cover of a graph $\mcG$. 
Let $\gamma$ be a closed walk $(v_0,e_1^{\sigma_1},\dots,v_{n})$ on $\mcG$, and let $\tilde{v}_0$ be a vertex in $\phi_V^{-1}(v_0)$. 
Then 
\[ \NV_\phi(\gamma^{-1},\tilde{v}_0)=\{g^{-1}\mid g\in \NV_\phi(\gamma,\tilde{v}_0)\}. \]
\end{lem}

\begin{proof}
Suppose that $g\in \NV_\phi(\gamma,\tilde{v}_0)$. 
There is a lift $\tilde{\gamma}:=(\tilde{v}_0,\tilde{e}_1^{\sigma_1},\dots,\tilde{v}_{n})$ of $\gamma$ such that $\tilde{v}_{n}=g{\cdot}\tilde{v}_0$. 
Since 
\[(g^{-1}{\cdot}\tilde{\gamma})^{-1}=(g^{-1}{\cdot}\tilde{v}_{n},(g^{-1}{\cdot}\tilde{e}_n)^{-\sigma_n},\dots,(g^{-1}{\cdot}\tilde{e}_1)^{-\sigma_1},g^{-1}{\cdot}\tilde{v}_0)\] 
is a walk on $\widetilde{\mcG}$, $g^{-1}\in \NV_\phi(\gamma^{-1},\tilde{v}_0)$. 
By the same argument, if $g\in \NV_\phi(\gamma^{-1},\tilde{v}_0)$, then $g^{-1}\in \NV_\phi(\gamma,\tilde{v}_0)$. 
Therefore, the assertion holds. 
\end{proof}

For two walks $\gamma_1=(v_0,\dots,v_{n})$ and $\gamma_2=(w_0,\dots,w_{m})$ on a graph $\mcG$ with $v_{n}=w_{0}$, 
let $\gamma_1\gamma_2$ denote the following walk of length $n+m$: 
\[ \gamma_1\gamma_2:=(v_0,\dots,v_{n},w_1,\dots,w_{m}). \]

\begin{lem}\label{lem: gap class of product}
Let $G$ be a finite group, and let $\phi:\widetilde{\mcG}\to\mcG$ be a $G$-cover of a graph $\mcG$. 
Let $\gamma_1$ and $\gamma_2$ be two closed walks on $\mcG$ with the same initial and terminal vertex $v_0$, and let $\tilde{v}_0$ be a vertex in $\phi_V^{-1}(v_0)$. 
Then 
\[ \NV_{\phi}(\gamma_1\gamma_2,\tilde{v}_0)=\NV_{\phi}(\gamma_1,\tilde{v}_0)\NV_{\phi}(\gamma_2,\tilde{v}_0), \]
where $S_1S_2=\{ g_1g_2\mid g_1\in S_1, g_2\in S_2 \}$ for two subsets $S_1, S_2 \subset G$. 
\end{lem}
\begin{proof}
Let $g$ be an element of $\NV_{\phi}(\gamma_1\gamma_2,\tilde{v}_0)$. 
Then we have $g{\cdot}\tilde{v}_0\in\widetilde{V}_{\phi,m+n}^{\gamma_1\gamma_2}(\tilde{v}_0)$. 
Hence there exists a lift $\tilde{\gamma}=(\tilde{v}_0,\dots,\tilde{v}_{m+n})$ of $\gamma_1\gamma_2$ with $\tilde{v}_{m+n}=g{\cdot}\tilde{v}_0$. 
Since $(\tilde{v}_0,\dots,\tilde{v}_{n})$ is a lift of $\gamma_1$, there is an element $g_1\in \NV_{\phi}(\gamma_1,\tilde{v}_0)$ such that $\tilde{v}_{n}=g_1{\cdot}\tilde{v}_0$. 
Since $\NV_{\phi}(\gamma_2,g_1{\cdot}\tilde{v}_0)=g_1\NV_{\phi}(\gamma_2,\tilde{v}_0)g_1^{-1}$ by Lemma~\ref{lem: conjugate} and $(\tilde{v}_{n},\dots,\tilde{v}_{m+n})$ is a lift of $\gamma_2$, there exists an element $g_2\in \NV_{\phi}(\gamma_2,\tilde{v}_0)$ such that 
\[ g{\cdot}\tilde{v}_0=\tilde{v}_{m+n}=(g_1g_2g_1^{-1}){\cdot}\tilde{v}_n=g_1g_2{\cdot}\tilde{v}_0. \]
Since $G$ acts on $\phi_V^{-1}(v_0)$ transitively, $\phi_V^{-1}(v_0)$ is isomorphic to the set of cosets $G/ H_1$ as $G$-sets, where $H_1=\{g'\in G\mid g'{\cdot}\tilde{v}_0=\tilde{v}_0\}$. 
Hence there is an element $g'\in H_1$ such that $g=g_1g_2g'$. 
Therefore, $g\in \NV_{\phi}(\gamma_1,\tilde{v}_0){\cdot}\NV_{\phi}(\gamma_2,\tilde{v}_0)$ since $g_2g'\in \NV_{\phi}(\gamma_2,\tilde{v}_0)$. 

Conversely, suppose that $g_i\in \NV_{\phi}(\gamma_i,\tilde{v}_0)$ for $i=1,2$. 
Then there are two lifts $\tilde{\gamma}_1=(\tilde{v}_0,\dots,\tilde{v}_{n})$ and $\tilde{\gamma}_2=(\tilde{w}_0,\dots,\tilde{w}_{m})$ of $\gamma_1$ and $\gamma_2$ such that $\tilde{w}_0=\tilde{v}_0$, $\tilde{v}_{n}=g_1{\cdot}\tilde{v}_0$ and $\tilde{w}_{m}=g_2{\cdot}\tilde{w}_0$. 
Then the ordered set 
\[ \tilde{\gamma}_1(g_1{\cdot}\tilde{\gamma}_2)=(\tilde{v}_0,\dots,\tilde{v}_{n},g_1{\cdot}\tilde{w}_1,\dots,g_1{\cdot}\tilde{w}_{m}) \]
is a lift of $\gamma_1\gamma_2$. 
Since $g_1{\cdot}\tilde{w}_{m}=g_1g_2{\cdot}\tilde{v}_0$, we obtain $g_1g_2\in \NV_{\phi}(\gamma_1\gamma_2,\tilde{v}_0)$. 
\end{proof}

For a closed walk $\gamma=(v_0,e_1^{\sigma_1},\dots,v_{n})$ on a graph $\mcG$ (hence $v_0=v_{n}$) and $j\in\ZZ_{\geq 0}$, let $\gamma^{(j)}$ be the following closed walk 
\[ \gamma^{(j)}:=(v_{\bar{j}},e_{\bar{j}+1}^{\sigma_{\bar{j}+1}},\dots,v_n,e_{1}^{\sigma_1},v_1,\dots,v_{\bar{j}}) \]
if $j\not\equiv 0\pmod{n}$
where $\bar{j}$ is the integer $0<\bar{j}< n$ with $\bar{j}\equiv j\pmod{n}$, 
and put $\gamma^{(j)}:=\gamma$ otherwise. 
In the case where the Galois group is abelian, the net voltage class of a closed walk does not depend on its initial vertex by the following lemma. 

\begin{lem}\label{lem: gap of shifted walk}
Let $G$ be a finite abelian group, let $\phi:\widetilde{\mcG}\to\mcG$ be a $G$-cover of graphs, and let $\gamma$ be a closed walk on $\mcG$. 
Then, for any $j\in\ZZ_{\geq 0}$,
\[ \NV_{\phi}(\gamma)=\NV_{\phi}(\gamma^{(j)}). \]
\end{lem}
\begin{proof}
Let $v_0$ be the initial vertex of $\gamma$. 
Since $G$ is abelian, $\NV_{\phi}(\gamma,\tilde{v}_0)$ does not depend on the choice of $\tilde{v}_0\in\phi_V^{-1}(v_0)$, and we can regard $\NV_\phi(\gamma)$ as a subset of $G$. 
Fix $\tilde{v}_0\in\phi_V^{-1}(v_0)$, and take an element $g\in \NV_{\phi}(\gamma)$. 
Let $\tilde{\gamma}=(\tilde{v}_0,\tilde{e}_1^{\sigma_1},\dots,\tilde{v}_{n})$ be a lift of $\gamma$ under $\phi$ such that $\tilde{v}_{n}=g{\cdot}\tilde{v}_0$. 
By Remark~\ref{rem: galois cover} (ii), the ordered set 
\[ \tilde{\gamma}':=(\tilde{v}_{\bar{j}},\tilde{e}_{\bar{j}+1}^{\sigma_{\bar{j}+1}},\dots,\tilde{v}_n,(g{\cdot}\tilde{e}_{1})^{\sigma_1},g{\cdot}\tilde{v}_1,\dots,g{\cdot}\tilde{v}_{\bar{j}}) \]
is a walk on $\widetilde{\mcG}$, hence it is a lift of $\gamma^{(j)}$. 
Thus $g\in \NV_{\phi}(\gamma^{(j)})$, and we obtain $\NV_{\phi}(\gamma)\subset \NV_{\phi}(\gamma^{(j)})$. 
Since $\gamma=(\gamma^{(j)})^{(n-\bar{j})}$, we also have $\NV_{\phi}(\gamma^{(j)})\subset \NV_{\phi}(\gamma)$. 
\end{proof}

In the case of abelian covers, we may regard a closed walk $\gamma$ as a class $[\gamma]:=\{\gamma^{(j)}\mid j\in\ZZ_{\geq0}\}$ in order to compute the net voltage class by Lemma~\ref{lem: gap of shifted walk}. 
The next example shows that Lemma~\ref{lem: gap of shifted walk} fails in the case where $G$ is not abelian.  

\begin{ex}
Let $\gtS_3$ be the symmetric group of three letters, and let $\mcG$ be the complete graph of order three. 
We construct a branched $\gtS_3$-cover of $\mcG$. 
Put $V_\mcG:=\{a,b,c\}$ and $E_\mcG:=\{e_{ab}, e_{bc}, e_{ca}\}$, where $e^+_{xy}(0)=x$ and $e^+_{xy}(1)=y$. 
Note that, since $\gtS_3$ acts transitively on $\phi_V^{-1}(v)$ for any $\gtS_3$-cover $\phi:\widetilde{\mcG}\to\mcG$ and each $v\in V_\mcG$, we can regard $\phi_V^{-1}(v)$ as the set of cosets of a certain subgroup of $\gtS_3$. 
We define a branched $\gtS_3$-cover $\phi:\widetilde{\mcG}\to\mcG$ of $\mcG$ as follows: 

Let $H_1$ and $H_2$ be the cyclic subgroups of $\gtS_3$ generated by $(1\,2)$ and $(1\, 3)$, respectively. 
Let $a_i$ and $b_i$ ($i=1,2,3$) be the cosets of $\gtS_3/H_1$ and $\gtS_3/H_2$, respectively, as follows: 
\begin{align*} 
&a_1=H_1,  &a_2=(1\, 3)H_1,&  &a_3=(2\, 3)H_1, \\ 
&b_1=H_2,  &b_2=(1\, 2)H_2, & &b_3=(2\, 3)H_2. 
\end{align*}
We put $c_i$ ($i=1,\dots,6$) as the elements of $\gtS_3$ as follows: 
\begin{align*}
c_1&= \id,  & c_2&=(1\,3), & c_3&=(1\,2), & c_4&=(1\,2\,3), & c_5&=(1\, 3\, 2) & c_6&=(2\,3). 
\end{align*}
Put $V_{\widetilde{\mcG}}:=\{ a_1,a_2,a_3, b_1,b_2,b_3, c_1,\dots,c_6 \}$, and define $\phi_V:V_{\widetilde{\mcG}}\to V_{\mcG}$ by 
$\phi_V(x_i)=x$ for $x=a,b,c$. 
Note that $\gtS_3$ acts on $\phi_V^{-1}(v)$ from left for each $v\in V_\mcG$. 
Assume that $e_{a_1b_1},e_{b_1c_1}, e_{c_1a_1}\in E_{\widetilde{\mcG}}$. 
Then the set $E_{\widetilde{\mcG}}$ of an $\gtS_3$-cover $\widetilde{\mcG}$ is determined by Definition~\ref{def: Galois cover} as Figure~\ref{fig: S3-cover}. 
Here, each arrowhead in Figure~\ref{fig: S3-cover} represents the plus direction of each edges. 

Let $\gamma$ be the closed walk $(a,e_{ab},b,e_{bc},c,e_{ca},a)$ on $\mcG$. 
We see that $\widetilde{V}_{\phi,4}^{\gamma}(a_1)=\{ a_1,a_2,a_3 \}=\phi_V^{-1}(a)$, 
hence we have $\NV_{\phi}(\gamma,a_1)=\gtS_3$. 
For the walk $\gamma^{(2)}=(c,e_{ca},a,e_{ab},b,e_{bc},c)$, we can see that $\widetilde{V}_4^{\gamma^{(2)}}(c_1)=\{ c_1,c_2,c_3,c_5 \}$. 
Hence we obtain 
\[ \NV_\phi(\gamma^{(2)},c_1)=\{ \id, (1\,3), (1\,2), (1\,3\,2) \}. \]
Therefore, $\NV_\phi(\gamma)\ne \NV_\phi(\gamma^{(2)})$. 
\begin{figure}
\begin{center}
\input{S3_cover_v3}
\caption{An $\gtS_3$-cover $\phi:\widetilde{\mcG}\to\mcG$}
\label{fig: S3-cover}
\end{center}
\end{figure}

\end{ex}

\begin{rem}\label{rem: abelian case}
Let $G$ be a finite abelian group, let $\phi:\widetilde{\mcG}\to\mcG$ be a $G$-cover of a graph $\mcG$, and let $\gamma=(v_0,e_1^{\sigma_1},\dots,v_{n})$ be a closed walk on $\mcG$. 
If $v_{i}=v_{j}$ for some $0\leq i<j< n$, then $\gamma^{(i)}$ splits into two closed walks: 
\[ \gamma^{(i)}=(v_{i},e_{i+1}^{\sigma_{i+1}},\dots,v_{j},e_{j+1}^{\sigma_{j+1}},\dots,v_{n},e_1^{\sigma_1},v_{1},\dots,v_{i})=\gamma_1\gamma_2, \]
where $\gamma_1=(v_{i},e_{i+1}^{\sigma_{i+1}},\dots,v_{j})$ and $\gamma_2=(v_{j},e_{j+1}^{\sigma_{j+1}},\dots,v_n,e_{1}^{\sigma_1},v_1,\dots,v_{i})$. 
By Lemma~\ref{lem: gap class of product} and \ref{lem: gap of shifted walk}, 
we have 
\[ \NV_{\phi}(\gamma)=\NV_{\phi}(\gamma_1)\NV_{\phi}(\gamma_2). \] 
Thus, in the case where $G$ is abelian, it is enough to compute $\NV_{\phi}(\gamma)$ for simple closed walks $\gamma$ in order to compute net voltage classes for all closed walks on $\mcG$. 
\end{rem}

\section{Splitting graphs and embedded topology}\label{sec: splitting graph}

In this section, we define the \textit{splitting graph} of plane curves for a Galois cover of $\PP^2$, and study a relation between the splitting graph and the embedded topology of plane curves. 
Here  a Galois cover of $\PP^2$ is a finite morphism $\phi:X\to\PP^2$ such that $X$ is normal and the extension of rational function fields $\CC(X)/\CC(\PP^2)$ is Galois. 
In the first subsection, we consider the splitting graph for general Galois covers. 
In the second subsection, we give a method of computing net voltage classes for the splitting graph in the case of cyclic covers with a certain assumption. 

\subsection{Splitting graphs for Galois covers}

We prepare some notation in order to define the splitting graph. 
Let $Y$ be a normal surface, let $\mcC$ be a reduced Weil divisor on $Y$, and 
let $P\in\mcC$ be a singular point of $\mcC$ such that $Y$ is smooth at $P$. 
We define $\Irr(\mcC)$ and $\LB_P(\mcC)$ as the sets of irreducible components of $\mcC$ and local branches of $\mcC$ at $P$, respectively. 
%Let $\LB_P(\mcC)$ be the set of local branches of $\mcC$ at $P$. 
Here a \textit{local branch} of $\mcC$ at $P$ is an irreducible component of the germ $(\mcC,P)$. 
For $C\in\Irr(\mcC)$ and $b\in\LB_P(\mcC)$, 
we say that \textit{$C$ contains $b$}, denote by $b\subset C$, if $b$ is an irreducible component of the germ $(C,P)$. 
By abuse of notation, for $b\in\LB_P(\mcC)$ and a morphism $\varphi:Y\to Y'$ with $Y'$ smooth at $\varphi(P)$, 
let $\varphi(b)$ denote the image of $b$ under the morphism $\varphi_P|_{(\mcC,P)}:(\mcC,P)\to(Y',\varphi(P))$ of germs. 
For a reduced Weil divisor $\mcB$ on $Y$ with $\Sing(Y)\subset\mcB$, 
put $\Sing(\mcC\setminus\mcB):=\{P\in\Sing(\mcC)\mid P\not\in\mcB\}$. 

\begin{defin}\label{def: splitting graph}
Let $G$ be a finite group, and let $\phi:X\to\PP^2$ be a $G$-cover of $\PP^2$. 
Let $\mcB_{\phi}$ be the branch locus of $\phi$, and let $\mcC\subset\PP^2$ be a plane curve such that $\mcC\cap\mcB_\phi$ is finite (equivalently, $\mcC$ and $\mcB_\phi$ have no common components). 
\begin{enumerate}
\item The \textit{incidence graph of $\mcC$ with respect to $\phi$} is the bipartite graph $\mcG=\mcG_{\phi,\mcC}$ such that the set $V_\mcG$ of vertices has the partition $(V_{\mcG,0},V_{\mcG,1})$, where 
\begin{align*} 
V_{\mcG,0}&:=\{v_P\mid P\in\Sing(\mcC\setminus\mcB_\phi)\} \ \mbox{ and } \\ 
V_{\mcG,1}&:=\{v_C\mid C\in\Irr(\mcC) \}; 
\end{align*}
and the set $E_{\mcG}$ of edges is 
\[ E_\mcG:=\bigcup_{P\in\Sing(\mcC\setminus\mcB_\phi)}\{e_{P,b} \mid b\in\LB_P(\mcC) \}, \]
where we define the initial and terminal vertices of the plus direction $e^+_{P,b}$ of $e_{P,b}\in E_{\mcG}$ by 
\begin{align*} 
e^+_{P,b}(0)&:=v_P\in V_{\mcG,0},   &  e^+_{P,b}(1)&:=v_C\in V_{\mcG,1} 
\end{align*}
for the irreducible component $C\in\Irr(\mcC)$ with $b\subset C$. 
\item The \textit{splitting graph of $\mcC$ for $\phi$} is the graph $\mcS:=\mcS_{\phi,\mcC}$ with 
the action of $G$ satisfying the following conditions;
	\begin{enumerate}
	\item the set $V_{\mcS}$ has the partition $(\widetilde{V}_{\mcS,0},\widetilde{V}_{\mcS,1})$, where 
	\begin{align*} 
	\widetilde{V}_{\mcS,0}&:=\bigcup_{P\in\Sing(\mcC\setminus\mcB_\phi)}\left\{v_{\widetilde{P}}\,\middle|\, \widetilde{P}\in\phi^{-1}(P)\right\},  \\ 
	\widetilde{V}_{\mcS,1}&:=\left\{ v_{\widetilde{C}}\,\middle|\, \widetilde{C}\in\Irr(\phi^\ast\mcC) \right\};  
	\end{align*}
	\item the set $E_{\mcS}$ of edges is 
	\[ E_{\mcS}:=\bigcup_{\widetilde{P}\in\phi^{-1}(\Sing(\mcC\setminus\mcB_\phi))}\left\{e_{\widetilde{P},\tilde{b}} \,\middle|\, \tilde{b}\in\LB_{\widetilde{P}}(\phi^\ast\mcC)\right\}, \] 
	where $e_{\widetilde{P},\tilde{b}}^+(0):=v_{\widetilde{P}}\in\widetilde{V}_{\mcS,0}$ and $e_{\widetilde{P},\tilde{b}}^+(1):=v_{\widetilde{C}}\in\widetilde{V}_{\mcS,1}$ for $\widetilde{C}\in\Irr(\phi^\ast\mcC)$ with $\tilde{b}\subset\widetilde{C}$; 
	\item $G$ acts on $\mcS$ via the image under the covering transformation $g:X\to X$ for each $g\in G$, 
	i.e., for $v_{\tilde{x}}\in V_{\mcS}$, $e_{\widetilde{P},\tilde{b}}\in E_{\mcS}$ and $g\in G$, we define $g{\cdot}v_{\tilde{x}}:=v_{g(\tilde{x})}$ and $g{\cdot}e_{\widetilde{P},\tilde{b}}:=e_{g(\widetilde{P}),g(\tilde{b})}$ (see Remark~\ref{rem: action on splitting graph}).  
\end{enumerate}
\end{enumerate}

\end{defin}

\begin{rem}\label{rem: action on splitting graph}
We define the action of $G$ on $\widetilde{V}_{\mcS,1}$ by using the image $\widetilde{C}\mapsto g(\widetilde{C})$ for $g\in G$ and $\widetilde{C}\in\Irr(\phi^\ast\mcC)$, NOT the pull-back $\widetilde{C}\mapsto g^\ast\widetilde{C}$. 
Since $\tilde{b}\subset\widetilde{C}$ if and only if $g(\tilde{b})\subset g(\widetilde{C})$ for $\tilde{b}\in\LB_{\widetilde{P}}(\phi^\ast\mcC)$ ($\widetilde{P}\in\phi^{-1}(\Sing(\mcC\setminus\mcB_\phi))$) and $\widetilde{C}\in\Irr(\phi^\ast\mcC)$, 
the endpoint sets satisfy $g{\cdot}(V_{\mcS}(e_{\widetilde{P},\tilde{b}}))=V_{\mcS}(g{\cdot}e_{\widetilde{P},\tilde{b}})$, hence $G$ acts on $\mcS$ by Definition~\ref{def: splitting graph} (ii-c). 
\end{rem}

\begin{lem}\label{lem: splitting graph map}
Let $\phi:X\to\PP^2$ be a $G$-cover, and let $\mcC\subset\PP^2$ be a plane curve such that $\mcC\cap\mcB_\phi$ is finite. 
Put the incidence graph $\mcG:=\mcG_{\phi,\mcC}$ and the splitting graph $\mcS:=\mcS_{\phi,\mcC}$. 
Let $\phi_{\mcC}=(\phi_{\mcC,V},\phi_{\mcC,E}):\mcS\to\mcG$ be the map defined by 
\begin{align*}
\phi_{\mcC,V}(v_{\tilde{x}})&:=v_{\phi(\tilde{x})}, &
\phi_{\mcC,E}(e_{\widetilde{P},\tilde{b}})&:=e_{\phi(\widetilde{P}),\phi(\tilde{b})}
\end{align*}
for $\tilde{x}\in\phi^{-1}(\Sing(\mcC\setminus\mcB_\phi))\cup\Irr(\phi^\ast\mcC)$, $\widetilde{P}\in\phi^{-1}(\Sing(\mcC\setminus\mcB_\phi))$ and $\tilde{b}\in\LB_{\widetilde{P}}(\phi^\ast\mcC)$. 
Then $\phi_\mcC$ is a $G$-cover of graphs. 
\end{lem}

\begin{proof}
Let $e_{\widetilde{P},\tilde{b}}$ be an edge of $\mcS$. 
Its endpointset is $V_\mcS(e_{\widetilde{P},\tilde{b}})=\{v_{\widetilde{P}},v_{\widetilde{C}}\}$ for $\widetilde{C}\in\Irr(\phi^\ast\mcC)$ with $\tilde{b}\subset \widetilde{C}$. 
Put $P:=\phi(\widetilde{P})$, $C:=\phi(\widetilde{C})$ and $b:=\phi(\tilde{b})$ the local branch of $\mcC$ at $P$. 
Since $b\subset C$, we have $V_\mcG(e_{P,b})=\{v_P,v_C\}$. 
Thus $\phi_{\mcC,V}(V_\mcS(e_{\widetilde{P},\tilde{b}}))=V_{\mcG}(\phi_{\mcC,E}(e_{\widetilde{P},\tilde{b}}))$, and $\phi_\mcC$ is a map of graphs. 

The group $G$ acts transitively on both of $\Irr(\phi^\ast C)$ and $\phi^{-1}(P)$ for $C\in\Irr(\mcC)$ and $P\in\Sing(\mcC\setminus\mcB_\phi)$. 
Hence $G$ acts transitively on both of $\phi_{V}^{-1}(v_C)$ and $\phi_V^{-1}(v_P)$, and we have $V_{\mcG}=V_{\mcS}/G$. 
Since $P\in\Sing(\mcC\setminus\mcB_\phi)$ is not a branch point of $\phi$, $G$ acts transitively and freely on the set 
\[ \{\tilde{b}\in\LB_{\widetilde{P}}(\phi^\ast\mcC)\mid \widetilde{P}\in\phi^{-1}(P), \phi(\tilde{b})=b\} \] 
for $b\in\LB_P(\mcC)$. 
Thus $G$ acts transitively and freely on $\phi_E^{-1}(e_{P,b})$ for $P\in\Sing(\mcC\setminus\mcB_\phi)$ and $b\in\LB_P(\mcC)$, and we obtain $E_{\mcG}=E_{\mcS}/G$. 
Therefore, $\phi$ is a $G$-cover of graphs. 
\end{proof}

Let $G$ be a finite group, and let $\mcB\subset\PP^2$ be a plane curve. 
It is known that a surjective homomorphism $\rho:\pi_1(\PP^2\setminus\mcB)\twoheadrightarrow G$ induces a $G$-cover $\phi:X\to\PP^2$, uniquely up to isomorphism over $\PP^2$, whose branch locus is contained in $\mcB$. 
Here $\pi_1(\PP^2\setminus\mcB)$ is the fundamental group of $\PP^2\setminus\mcB$. 
Conversely, a $G$-cover $\phi:X\to\PP^2$ branched at $\mcB$ induces a surjective homomorphism $\rho:\pi_1(\PP^2\setminus \mcB)\twoheadrightarrow G$. 
We roughly recall the surjection $\rho:\pi_1(\PP^2\setminus\mcB)\twoheadrightarrow G$ induced by a $G$-cover $\phi:X\to\PP^2$ (cf. \cite{namba} for details). 

Let $\phi:X\to\PP^2$ be a $G$-cover branched at $\mcB$, and we fix a base point $P_0\in\PP^2\setminus\mcB$. 
Put $U:=\PP^2\setminus \mcB$ and $\widetilde{U}:=X\setminus\phi^{-1}(\mcB)$. 
Any element $[\gamma]$ of $\pi_1(U)=\pi_1(U,P_0)$ can be represented by a closed path $\gamma:[0,1]\to U$ with $\gamma(0)=\gamma(1)=P_0$. 
For a point $\widetilde{P}\in \widetilde{U}$, put $P:=\phi(\widetilde{P})$. 
Let $p:[0,1]\to U$ be a path from $P$ to $P_0$. 
Then the closed path $p^{-1}\gamma p$ is uniquely lifted in a path $\lambda:[0,1]\to \widetilde{U}$ with $\lambda(0)=\widetilde{P}$. 
It is known that the point $\lambda(1)$  depends only on the choice of $[\gamma]$ and $\widetilde{P}$, 
and that $[\gamma]$ gives an isomorphism $g_{\gamma}:\widetilde{U}\to\widetilde{U}$ defined by $g_{\gamma}(\widetilde{P})=\lambda(1)$. 
Then this correspondence gives a surjective homomorphism $\pi_1(U)\to \Aut_{\phi}(\widetilde{U})$ defined by $[\gamma]\mapsto g_{\gamma}$, 
where $\Aut_{\phi}(\widetilde{U})$ is the group $\{ g\in\Aut(\widetilde{U}) \mid \phi\circ g=\phi \}$. 
Since $\Aut_\phi(\widetilde{U})$ is isomorphic to $G$, we obtain a surjective homomorphism $\rho:\pi_1(U)\to G$. 

\begin{thm}\label{thm: splitting graph}
Let $G$ be a finite group, and let $\mcB_i$ ($i=1,2$) be two plane curves such that there are surjections $\rho_i:\pi_1(\PP^2\setminus\mcB_i)\twoheadrightarrow G$. 
For each $i=1,2$, let $\phi_i:X_i\to\PP^2$ be a $G$-cover branched at $\mcB_i$ induced by $\rho_i$, and
let $\mcC_i$ be a plane curve such that $\mcC_i\cap\mcB_i$ is finite. 
Assume that there exists a homeomorphism $h:\PP^2\to\PP^2$ such that $h(\mcB_1)=\mcB_2$, $h(\mcC_1)=\mcC_2$ and $\rho_2\circ h_{\ast}=\tau\circ\rho_1$ for some automorphism $\tau:G\to G$, where $h_{\ast}:\pi_1(\PP^2\setminus\mcB_1) \to \pi_1(\PP^2\setminus\mcB_2)$ is the isomorphism induced by $h$. 
Put $\mcG_i:=\mcG_{\phi_i,\mcC_i}$ and $\mcS_i:=\mcS_{\phi_i,\mcC_i}$. 
Then the following conditions hold:
\begin{enumerate}
\item $h$ induces an isomorphism $\theta_h:\mcG_{1}\to\mcG_{2}$ of the incidence graphs preserving the partitions, i.e., $\theta_{h,V}(V_{\mcG_1,j})=V_{\mcG_2,j}$ for $j=0,1$, where $(V_{\mcG_i,0},V_{\mcG_i,1})$ is the partition of $V_{\mcG_i}$ in Definition~\ref{def: splitting graph} (i); 
\item
the splitting graphs $\mcS_{1}$ and $\mcS_{2}$ are $(\theta_h,\tau)$-equivalent, \\ $\mcS_{1}\sim_{(\theta_h,\tau)}\mcS_{2}$.
\end{enumerate}
\end{thm}
\begin{proof}
We define $\theta_{h,V}:V_{\mcG_1}\to V_{\mcG_2}$ and $\theta_{h,E}:E_{\mcG_1}\to E_{\mcG_2}$ by 
\begin{align*}
\theta_{h,V}(v_{x_1})&:=v_{h(x_1)}, & \theta_{h,E}(e_{P_1,b_1})&:=e_{h(P_1),h(b_1)},
\end{align*}
respectively, for $x_1\in\Sing(\mcC_1\setminus\mcB_1)\cup\Irr(\mcC_1)$, $P_1\in\Sing(\mcC_1\setminus\mcB_1)$ and $b_1\in\LB_{P_1}(\mcC_1)$. 
Since $h(\mcB_1)=\mcB_2$ and $h(\mcC_1)=\mcC_2$, $\theta_{h,V}$ is well-defined, bijective and preserving the partitions. 
The map $\theta_{h,E}$ is well-defined and bijective since $b_1\in\LB_{P_1}(\mcC_1)$ if and only if $h(b_1)\in\LB_{h(P_1)}(\mcC_2)$. 
Moreover, 
we have $V_{\mcG_2}(\theta_{h,E}(e_{P_1,b_1}))=\theta_{h,V}(V_{\mcG_1}(e_{P_1,b_1}))$ 
 since $b_1\subset C_1$ if and only if $h(b_1)\subset h(C_1)$ for $C_1\in\Irr(\mcC_1)$. 
Therefore, $\theta_h=(\theta_{h,V},\theta_{h,E})$ is an isomorphism of the incidence graphs $\mcG_{i}$, and assertion (i) holds. 

We denote $\PP^2\setminus\mcB_i$ and $X_i\setminus\phi_i^{-1}(U_i)$ by $U_i$ and $\widetilde{U}_i$, respectively, for each $i=1,2$. 
By the uniqueness of unramified $G$-covers induced by $\rho_i$, there is a homeomorphism $\tilde{h}:\widetilde{U}_1\to\widetilde{U}_2$ satisfying $h|_{U_1}\circ\phi_1|_{\widetilde{U}_1}=\phi_2|_{\widetilde{U}_2}\circ\tilde{h}$. 
Hence we have the following commutative diagram: 
\[ \begin{diagram}
\node{\widetilde{U}_1} \arrow{e,t}{\tilde{h}} \arrow{s,l}{\phi_1|_{\widetilde{U}_1}}  \node{\widetilde{U}_2} \arrow{s,r}{\phi_2|_{\widetilde{U}_2}} \\
\node{U_1} \arrow{e,t}{h|_{U_1}} \node{U_2}
\end{diagram} \]
By the same argument of proof of (i), we have an isomorphism $\tilde{\theta}_{h}:\mcS_{1}\to\mcS_{2}$ of graphs. 
From the above diagram, it is easy to see that $\theta_{h,\bullet}\circ\phi_{\mcC_1,\bullet}=\phi_{\mcC_2,\bullet}\circ\tilde{\theta}_{h,\bullet}$ for each $\bullet=V, E$. 
Thus condition (i) in Definition~\ref{def: equivalence of Galois covers} holds. 

Take a point $\widetilde{P}_1\in\widetilde{U}_1$. 
Let $\gamma:[0,1]\to U_1$ be a closed path with $\gamma(0)=\phi_1(\widetilde{P}_1)$, and let $\tilde{\gamma}:[0,1]\to\widetilde{U}_1$ be the lift of $\gamma$ satisfying $\tilde{\gamma}(0)=\widetilde{P}_1$. 
Then $g_{\gamma}{\cdot}\widetilde{P}_1=\widetilde{Q}_1:=\tilde{\gamma}(1)$, where $g_{\gamma}=\rho_1([\gamma])$. 
Note that $h_\ast([\gamma])$ is the element of $\pi_1(U_2)$ represented by the path $h\circ\gamma:[0,1]\to U_2$, and 
that $\tilde{h}\circ\tilde{\gamma}:[0,1]\to \widetilde{U}_2$ is the lift of $h\circ\gamma$ whose initial point is $\widetilde{P}_2:=\tilde{h}(\widetilde{P}_1)$, i.e., $\tilde{h}\circ\tilde{\gamma}(0)=\widetilde{P}_2$. 
Hence we have 
$g_{h\circ\gamma}{\cdot}\widetilde{P}_2=\tilde{h}(\widetilde{Q}_1)$, 
where $g_{h\circ\gamma}=\rho_2([h\circ\gamma])$. 
Since $\rho_2\circ h_{\ast}=\tau\circ\rho_1$ by the assumption, we have
\[ g_{h\circ\gamma}=\rho_2\circ h_\ast([\gamma])=\tau\circ\rho_1([\gamma])=\tau(g_{\gamma}). \]
Hence we obtain 
\[ \tilde{h}(g_{\gamma}{\cdot}\widetilde{P}_1)=\tilde{h}(\widetilde{Q}_1)=g_{h\circ\gamma}\cdot\widetilde{P}_2=\tau(g_{\gamma}){\cdot}\tilde{h}(\widetilde{P}_1) \]
for any $\widetilde{P}_1\in\widetilde{U}_1$ and any closed path $\gamma$ with $\gamma(0)=\phi_1(\widetilde{P}_1)$. 
By condition (ii-c) in Definition~\ref{def: splitting graph}, we have $\tilde{\theta}_{h,V}(g{\cdot}\tilde{v})=\tau(g){\cdot}\tilde{\theta}_{h,V}(\tilde{v})$ and $\tilde{\theta}_{h,E}(g{\cdot}\tilde{e})=\tau(g){\cdot}\tilde{\theta}_{h,E}(\tilde{e})$ for any $\tilde{v}\in V_{\mcS_{1}}$, $\tilde{e}\in E_{\mcS_1}$ and $g\in G$. 
Therefore we conclude that $\mcS_{1}\sim_{(\theta_h,\tau)}\mcS_{2}$. 
\end{proof}

\begin{rem}\label{rem: fundamental group}
Let $\mcC_i+\mcB_i$ ($i=1,2$) be two plane curves, let $\phi_i:X_i\to\PP^2$ be $G$-covers, and put $\mcG_i:=\mcG_{\phi_i,\mcC_i}$ and $\mcS_i:=\mcS_{\phi_i,\mcC_i}$ as in Theorem~\ref{thm: splitting graph}. 
\begin{enumerate}
\item
If there exist tubular neighborhoods $\mcT(\mcC_i+\mcB_i)\subset\PP^2$ of $\mcC_i+\mcB_i$ ($i=1,2$) such that there is a homeomorphism $h':\mcT(\mcC_1+\mcB_1)\to\mcT(\mcC_2+\mcB_2)$ with $h'(\mcC_1)=\mcC_2$ and $h'(\mcB_1)=\mcB_2$, 
then $h'$ induces an isomorphism $\theta_{h'}:\mcG_{1}\to\mcG_{2}$ preserving the partitions as the proof of Theorem~\ref{thm: splitting graph} (i). 
\item
In the case where $\mcC_i$ are irreducible, the cardinality of $\widetilde{V}_{\mcS_{i},1}\subset V_{\mcS_{i}}$ is equal to the splitting number $s_{\phi_i}(\mcC_i)$, which is the number of irreducible components of $\phi^\ast\mcC_i$. 
Hence, if $s_{\phi_1}(\mcC_1)\ne s_{\phi_2}(\mcC_2)$, then there is no isomorphism $\theta:\mcG_{1}\to\mcG_{2}$ preserving the partitions such that $\mcS_{1}\sim_{(\theta,\tau)}\mcS_{2}$ for any automorphism $\tau:G\to G$. 
Since the splitting number is not determined by the fundamental group by \cite{shirane}, the splitting graph is not determined by the one. 
\item 
It is easy to see that the number of connected components of $\mcS_{i}$ is equal to the connected number $c_{\phi_i}(\mcC_i)$ (cf. \cite{shirane2}). 
\end{enumerate}

\end{rem}

In order to restrict the possibility of $\tau:G\to G$ in Theorem~\ref{thm: splitting graph}, we discuss the image of a meridian of an irreducible component of a plane curve $\mcB$ under a homeomorphism $h:\PP^2\to\PP^2$. 
Let $P\in \mcB$ be a smooth point of $\mcB$. 
Take an open neighborhood $U\subset\PP^2$ of $P$ and a system of local coordinates $(x,y)$ of $U$ with $P=(0,0)$ so that $\mcB$ is defined by $x=0$. 
Let $\delta_{\epsilon}$ be the closed path $[0,1]\to U$ defined by $t\mapsto (\epsilon\exp(2\pi\sqrt{-1}\,t), 0)$ for a small number $\epsilon>0$. 
We call $m:=p\delta_{\epsilon}p^{-1}$ a \textit{meridian} of $\mcB$ at $P$, where $p:[0,1]\to\PP^2$ is a path from the base point $\ast\in\PP^2\setminus\mcB$ to the point $(\epsilon,0)\in U$. 
By abuse of notation, the class $[m]\in\pi_1(\PP^2\setminus\mcB)$ be also called the \textit{meridian} of $\mcB$ at $P$. 
Note that meridians $m$ and $m'$ at $P$ and $P'$, respectively, are homotopically equivalent in $\PP^2\setminus\mcB$ if $P$ and $P'$ are contained in an irreducible component of $\mcB$. 
The next lemma is effective to restrict the possibility of a automorphism $\tau:G\to G$ in Theorem~\ref{thm: splitting graph}. 

\begin{lem}\label{lem: meridian}
Let $\mcB_1, \mcB_2$ be two plane curves, 
and let $P_1\in B_1$ be a smooth point of $\mcB_1$. 
Assume that there is a homeomorphism $h:\PP^2\to\PP^2$ such that $h(\mcB_1)=\mcB_2$. 
Put $P_2:=h(P_1)$. 
Let $m_i:[0,1]\to\PP^2$ be a meridian of $\mcB_i$ at $P_i$ for each $i=1,2$. 
Then the closed path $h\circ m_1$ is homotpically equivalent to either a certain meridian of $\mcB_2$ at $P_2$ or its inverse in $\PP^2\setminus\mcB_2$. 
Equivalently, the class $[h\circ m_1]$ is a conjugate of either $[m_2]$ or $[m_2]^{-1}$ in $\pi_1(\PP^2\setminus\mcB_2)$ for any meridian $m_2$ of $\mcB_2$ at $P_2$. 
\end{lem}

\begin{proof}
Fix base points $\ast_1\in\PP^2\setminus\mcB_1$ and $\ast_2:=h(\ast_1)$ of $\PP^2\setminus\mcB_1$ and $\PP^2\setminus\mcB_2$, respectively. 
For each $i=1,2$, let $U_i\subset\PP^2$ be an open neighborhood of $P_i$, and 
let $(x_i,y_i)$ be a system of local coordinates of $U_i$ with $P_i=(0,0)$ so that $\mcB_i$ is defined by $x_i=0$. 
Let $\epsilon_2>0$ be a small number such that $m_2=p_2\delta_{\epsilon_2}p_2^{-1}$, where $\delta_{\epsilon_2}$ is the closed path in $U_2$, and $p_2$ is a path from $\ast_2$ to $(\epsilon_2,0)\in U_2$ in $\PP^2\setminus\mcB_2$. 
For a positive number $\epsilon>0$, put 
\[ D_{i,\epsilon}:=\{(x_i,y_i)\mid |x_i|^2+|y_i|^2\leq\epsilon^2\}\subset U_i. \]
Since the preimage $h^{-1}(\mathrm{Int}(D_{2,\epsilon_2}))$ of the interior $\mathrm{Int}(D_{2,\epsilon_2})$ of $D_{2,\epsilon_2}$ is open in $U_1$, 
there exists a positive number $\epsilon_1>0$ such that $D_{1,\epsilon_1}\subset h^{-1}(\mathrm{Int}(D_{2,\epsilon_2}))$. 
We may assume that $\delta_{\epsilon_1}$ with respect to $(x_1,y_1)$ satisfies that $m_1=p_1 \delta_{\epsilon_1} p_1^{-1}$, where $p_1$ is a path from the base point $\ast_1$ to $(\epsilon_1,0)\in U_1$ in $\PP^2\setminus\mcB_1$. 
Since $h(\mathrm{Int}(D_{1,\epsilon_1}))$ is open in $U_2$, there exists a positive number $\epsilon_2'<\epsilon_2$ such that $D_{2,\epsilon_2'}\subset h(\mathrm{Int}(D_{1,\epsilon_1}))$. 
Since $h(\mcB_1)=\mcB_2$, the inclusions $D_{2,\epsilon_2'}\to D_{2,\epsilon_2}$, $D_{2,\epsilon_2'}\to h(D_{1,\epsilon_1})$ and $h(D_{1,\epsilon_1})\to D_{2,\epsilon_2}$ induce the morphisms 
\begin{align*}
i_{1\ast}&:\pi_1(D_{2,\epsilon_2'}\setminus\mcB_2)\to\pi_1(D_{2,\epsilon_2}\setminus\mcB_2), \\
i_{2\ast}&:\pi_1(D_{2,\epsilon_2'}\setminus\mcB_2)\to\pi_1(h(D_{1,\epsilon_1}\setminus\mcB_1)) \ \mbox{ and } \\
i_{3\ast}&:\pi_1(h(D_{1,\epsilon_1}\setminus\mcB_1))\to \pi_1(D_{2,\epsilon_2}\setminus\mcB_2), 
\end{align*}
respectively. 
Note that we have 
\[
\pi_1(D_{2,\epsilon_2'}\setminus\mcB_2)\cong\pi_1(D_{2,\epsilon_2}\setminus\mcB_2)\cong\pi_1(D_{1,\epsilon_1}\setminus\mcB_1)\cong\ZZ.
\]
Since $i_{1\ast}=i_{3\ast}\circ i_{2\ast}$ and $i_{1\ast}$ is isomorphic, 
the composition of $h_\ast$ and $i_{3\ast}$ maps a generator of $\pi_1(D_{1,\epsilon_1}\setminus\mcB_1)$ to a generator of $\pi_1(D_{2,\epsilon_2}\setminus\mcB_2)$. 
Thus $h\circ\delta_{\epsilon_1}$ is homotopically equivalent to either $\delta_{\epsilon_2}$ or $\delta_{\epsilon_2}^{-1}$ in $\PP^2\setminus\mcB_2$. 
This implies that $h_\ast([m_1])$ is a conjugate of either $[m_2]$ or $[m_2]^{-1}$.  
\end{proof}

\begin{rem}\label{rem: smooth curve}
If $\mcB_1$ and $\mcB_2$ are smooth plane curves of degree $d$, then $\pi_1(\PP^2\setminus\mcB_i)\cong\ZZ_d$. 
For a meridian $m_{\mcB_i}$ of $\mcB_i$, we assume that the class $[m_{\mcB_i}]\in\pi_1(\PP^2\setminus\mcB_i)$ corresponds to the generator $[1]\in\ZZ_d$. 
By Lemma~\ref{lem: meridian}, a homeomorphism $h:\PP^2\to\PP^2$ with $h(\mcB_1)=\mcB_2$ induces either $\tau_d^+:\ZZ_d\to\ZZ_d$ or $\tau_d^-:\ZZ_d\to\ZZ_d$, 
where the automorphisms $\tau_d^\pm$ are defined by $\tau_d^\pm([1])=[\pm 1]$, respectively. 
\end{rem}

We define an equivalence between splitting graphs as follows. 

\begin{defin}\label{def: equivalence of splitting graph}
Let $G$ be a finite group, and let $\mcB_i$ ($i=1,2$) be two plane curves such that there are surjections $\rho_i:\pi_1(\PP^2\setminus\mcB_i)\to G$. 
Let $\phi_i:X_i\to\PP^2$ be a $G$-cover induced by $\rho_i$, and 
let $\mcC_i$ be a plane curve such that $\mcC_i\cap\mcB_i$ is finite for each $i=1,2$. 
The splitting graphs $\mcS_{\phi_1,\mcC_1}$ and $\mcS_{\phi_2,\mcC_2}$ are said to be \textit{equivalent}, denoted by $\mcS_{\phi_1,\mcC_1}\sim\mcS_{\phi_2,\mcC_2}$, 
if there exist a homeomorphism $h':\mcT(\mcC_1+\mcB_1)\to\mcT(\mcC_2+\mcB_2)$ of tubular neighborhoods $\mcT(\mcC_i+\mcB_i)\subset\PP^2$ of $\mcC_i+\mcB_i$ and an automorphism $\tau:G\to G$ satisfying 
\begin{enumerate}
\item $h'(\mcC_1)=\mcC_2$ and $h'(\mcB_1)=\mcB_2$; 
\item for any meridian $m_B$ of any irreducible component $B\subset\mcB_1$, either $\tau(\rho_1([m_B]))=\rho_2([m_{h'(B)}])$ or $\tau(\rho_1([m_B]))=\rho_2([m_{h'(B)}]^{-1})$ for some meridian $m_{h'(B)}$ of $h'(B)$; and
\item $\mcS_{\phi_1,\mcC_1}\sim_{(\theta_{h'},\tau)}\mcS_{\phi_2,\mcC_2}$ as $G$-covers of graphs, where $\theta_{h'}:\mcG_{\phi_1,\mcC_1}\to\mcG_{\phi_2,\mcC_2}$ is the isomorphism in Remark~\ref{rem: fundamental group} (i). 
\end{enumerate}
\end{defin}

\begin{rem}
A homeomorphism $h':\mcT(\mcC_1+\mcB_1)\to\mcT(\mcC_2+\mcB_2)$ in Definition~\ref{def: equivalence of splitting graph} gives a correspondence between the combinatorial data of $\mcC_1+\mcB_1$ and $\mcC_2+\mcB_2$, which consist of the sets of irreducible components, singularities, degrees of components, and configuration of components of $\mcC_i+\mcB_i$. 
Conversely, it is known that a correspondence between the combinatorial data induces a homeomorphism $h'$ between tubular neighborhoods (cf. \cite[Remark~3]{acc}). 
\end{rem}

By Theorem~\ref{thm: splitting graph} and Lemma~\ref{lem: meridian}, we obtain the following lemma. 

\begin{cor}\label{cor: splitting graph}
Under the assumption of Theorem~\ref{thm: splitting graph}, 
the splitting graphs $\mcS_{\phi_1,\mcC_1}$ and $\mcS_{\phi_2,\mcC_2}$ are equivalent. 
\end{cor}

\begin{ex}
Let $\mcB$ be the conic defined by $z^2-4x y=0$, and 
let $\mcC_1$ and $\mcC_2$ be two $6$-nodal irreducible sextics in \cite[Example~6.2]{bannai-shirane} and \cite[Example~6.3]{bannai-shirane}, respectively. 
Note that $\mcB$ is a simple contact conic of both of $\mcC_1$ and $\mcC_2$. 
Let $\phi:X\to\PP^2$ be the double cover branched at the conic $\mcB$. 
Then $\mcC_1$ is a splitting curve with respect to $\phi$, write $\phi^\ast\mcC_1=\widetilde{\mcC}_1^++\widetilde{\mcC}_1^-$, and $\phi^{-1}(\Sing(\mcC_1\setminus\mcB))=(\widetilde{\mcC}_1^+\cap\widetilde{\mcC}_1^-)\setminus\phi^{-1}(\mcB)$. 
On the other hand, $\widetilde{\mcC}_2:=\phi^\ast\mcC_2$ is irreducible. 
Thus the preimage of the $6$ nodes of $\mcC_2$ are the $12$ nodes of $\widetilde{\mcC}_2$. 
Hence the splitting graphs $\phi_{\mcC_i}:\mcS_{\phi,\mcC_i}\to\mcG_{\phi,\mcC_i}$ are as Figure~\ref{fig: example1} and \ref{fig: example2}, respectively, since a node consists of two local branches, 
where $v_{ij}$ are vertices corresponding to the $6$ nodes of $\mcC_i$, and $\phi_{\mcC_i}^{-1}(v_{ij})=\{v_{ij}^+,v_{ij}^-\}$. 
Hence $\mcS_{\phi_1,\mcC_1}$ and $\mcS_{\phi_2,\mcC_2}$ are not equivalent. 

\begin{figure}
\begin{center}
\input{example1}
\caption{The splitting graph of $\mcC_1$ for $\phi$}
\label{fig: example1}
\end{center}
\end{figure}

\begin{figure}
\begin{center}
\input{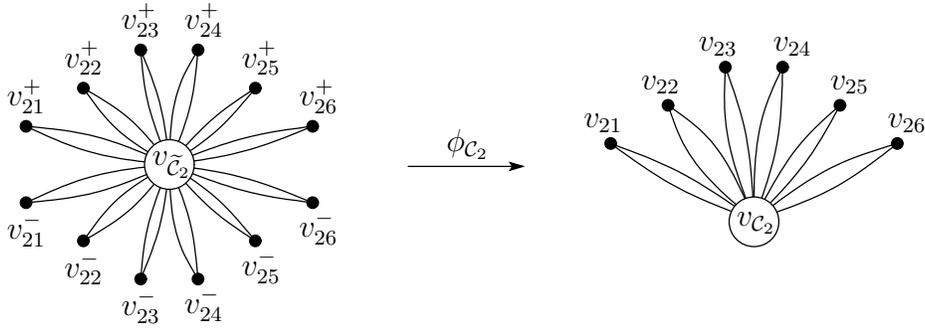}
\caption{The splitting graph of $\mcC_2$ for $\phi$}
\label{fig: example2}
\end{center}
\end{figure}

\end{ex}

\subsection{Net voltage classes of Splitting graphs for cyclic covers}

Proposition~\ref{prop: gap class} and Theorem~\ref{thm: splitting graph} imply that computing net voltage classes of closed walks is effective to distinguish the embedded topology of plane curves. 
We investigate a computation of net voltage classes for cyclic covers. 
Let $\overline{\mcB}$ and $\mcB$ be divisors 
\[ \overline{\mcB}=\sum_{i=1}^{m-1} i{\cdot}B_i \ \mbox{ and } \ \mcB=\sum_{i=1}^{m-1}B_i \]
on $\PP^2$, respectively, 
where $B_i$ ($i=1,\dots,m-1$) are reduced divisors with no common components each other. 
Note that the degree of $\overline{\mcB}$ is divisible by $m$ 
if and only if 
there exists a surjection $\rho:\pi_1(\PP^2\setminus\mcB)\twoheadrightarrow\ZZ_m:=\ZZ/m\ZZ$ which sends any meridian of $\mcB$ at any $P_i\in B_i\setminus\Sing(\mcB)$ to the image $[i]\in\ZZ_m$ of $i$ in $\ZZ_m$ (cf. \cite{pardini}). 
Assume that the degree of $\overline{\mcB}$ is divisible by $m$. 
We call the cyclic cover induced by the surjection $\rho:\pi_1(\PP^2\setminus\mcB)\twoheadrightarrow\ZZ_m$ as above the \textit{$\ZZ_m$-cover of type $\overline{\mcB}$}. 

Let $\phi:X\to\PP^2$ be the $\ZZ_m$-cover of type $\overline{\mcB}$, and 
let $\mcC\subset\PP^2$ be a plane curve such that $\mcC\cap\mcB$ is finite. 
We assume the following condition to compute the net voltage class of a closed walk for $\phi$: 
\begin{align}\label{eq: assumption}
\mbox{All irreducible components of $\mcC$ are smooth.} 
\end{align}

\begin{rem}
Under assumption~(\ref{eq: assumption}), the incidence graph $\mcG_{\phi,\mcC}$ has no parallel edge. 
Hence an edge of $\mcG_{\phi,\mcC}$ is identified with a pair $(v_P,v_C)$ of vertices $v_P\in V_{\mcG_{\phi,\mcC},0}$ and $v_C\in V_{\mcG_{\phi,\mcC},1}$. 
In this case, we omit edges from sequences representing walks on $\mcG_{\phi,\mcC}$. 
Namely, we represent walks by sequences of vertices only. 
\end{rem}

Let $L\subset\PP^2$ be a line which intersects transversally with $\mcC$, and is not a component of $\mcB$. 
Since $L$ does not pass through singularities of $\mcC$, it is enough to consider the singular points and irreducible components of $\mcC$ and $\phi^{\ast}\mcC$ over the affine open set $U':=\PP^2\setminus L$ for computing the net voltage class. 
Hence we consider the restriction $\phi':\widetilde{U}'\to U'$ of $\phi$ to $\widetilde{U}':=X\setminus\phi^{-1}(L)$. 
We regard the coordinate ring of $U'$ as the polynomial ring $\CC[x,y]$. 
Let $F=0$ be a defining equation of $\overline{\mcB}$ on $U'$. 
By $L\not\subset\mcB$ and the proof of \cite[Theorem~2.1]{benoit-shirane} (cf. \cite[Theorem~2.7]{shirane}), 
if $C\in\Irr(\mcC)$ is defined by $f=0$ on $U'$, and if the splitting number of $C$ for $\phi$ is $s$, 
then there are two polynomials $g,h\in\CC[x,y]$ satisfying the following equation: 
\[
 F=fg+h^s. 
\]

Let $\gamma$ be the following closed walk on the incidence graph $\mcG_{\mcC}$: 
\[ 
\gamma=(v_{P_1},v_{C_1},v_{P_2},v_{C_2},\dots,v_{P_{n}},v_{C_n},v_{P_{n+1}}), 
\]
where $P_i\in\Sing(\mcC\setminus\mcB)$ with $P_{n+1}=P_1$ and $C_i\in\Irr(\mcC)$. 
We fix a defining equation $f_i=0$ of $C_i$ and polynomials $g_i,h_i\in\CC[x,y]$ satisfying 
\begin{align}\label{eq: g and h}
F=f_ig_i+h_i^{s_i} 
\end{align}
for each $i=1,\dots,n$, where $s_i$ is the splitting number of $C_i$ for $\phi$. 
Since $P_{i+1}$ is an intersection of $C_{i}$ and $C_{i+1}$ for $i=1,\dots,n$, 
we have $F(P_{i+1})=h_{i}^{s_{i}}(P_{i+1})=h_{i+1}^{s_{i+1}}(P_{i+1})$, where $C_{n+1}:=C_1$. 
For each $i=1,\dots, n$, we fix a complex number $d_i\in\CC$ such that 
\[
h_i(P_i)=d_i^{\mu_i}, 
\]
where $\mu_i:=m/s_i$. 
Since $h_{i}^{s_{i}}(P_{i+1})=h_{i+1}^{s_{i+1}}(P_{i+1})$, there is an integer $\alpha_i$ with $0\leq \alpha_i<s_{i}$ such that 
\begin{align}\label{eq: integer}
h_{i}(P_{i+1})=(\zeta_m^{\mu_{i}})^{\alpha_i}d_{i+1}^{\mu_{i}} 
\end{align}
for each $i=1,\dots,n$, where $\zeta_m:=\exp(2\pi\sqrt{-1}/m)$ and $d_{n+1}=d_1$. 
We put $\alpha:=\sum_{i=1}^n\alpha_i$. 

\begin{thm}\label{thm: gap class}
Under the above circumstance, the following equation holds:
\[ \NV_\phi(\gamma)=[\alpha]+s \ZZ_m:=\{ [\alpha+sk]\in\ZZ_m\mid k\in\ZZ \}. \]
\end{thm}
\begin{proof}
Let $\widetilde{U}''$ be the subvariety of $U'\times\CC$ defined by 
\[
t^m=F, 
\]
where $t$ is a coordinate of $\CC$, 
and let $\phi'':\widetilde{U}''\to U'$ be the projection. 
Note that $\widetilde{U}'$ is the normalization of $\widetilde{U}''$. 
Moreover, the action of $\ZZ_m$ on $\widetilde{U}''$ is given by 
\begin{equation}\label{eq: action} 
[1]{\cdot}(P,\zeta_m^{j}d_P)=(P,\zeta_m^{j+1}d_P),  
\end{equation}
where $[1]$ denotes the image of $1\in\ZZ$ in $\ZZ_m$ (cf. Remark~\ref{rem: meridian 2}), and $d_P$ is a complex number with $d_P^{m}=F(P)$. 
Since $\widetilde{U}''$ is smooth over $U'\setminus\mcB$, we have $\widetilde{U}'\setminus(\phi')^{-1}(\mcB)\cong\widetilde{U}''\setminus(\phi'')^{-1}(\mcB)$. 
Thus it is enough to consider $(\phi'')^{\ast}\mcC$ to compute the net voltage class. 

Since $F(P_i)=d_i^m$ and $\widetilde{U}''$ is defined by $t^m=F$ in $U'\times\CC$, the preimage of $P_i$ under $\phi'':\widetilde{U}''\to U'$ consists of the following $m$ points 
\[ \widetilde{P}_{i,j}:=(P_i, \zeta_m^{j}d_i)\in\widetilde{U}''\subset U'\times \CC \ \ \  (j=0,\dots,m-1) \]
for $i=1,\dots,n+1$. 
Note that $\widetilde{P}_{n+1,j}=\widetilde{P}_{1,j}$ since $P_{n+1}=P_1$ and $d_{n+1}=d_1$. 
Let $\widetilde{C}_{i,k}$ be the irreducible component of $(\phi'')^{\ast}C_i$ defined by  the following equation in $U'\times\CC$; 
\[ \widetilde{C}_{i,k} : t^{\mu_i}-(\zeta_m^{\mu_i})^{k}h_i=f_i=0 \]
for each $i=1,\dots,n$ and $k=0,\dots,s_i-1$. 

\begin{claim}\label{claim: 1}
\begin{enumerate}
\item $\widetilde{P}_{i,j}\in\widetilde{C}_{i,k}$ if and only if $j\equiv k\pmod{s_i}$. 
\item $\widetilde{P}_{i+1,j}\in\widetilde{C}_{i,k}$ if and only if $j\equiv \alpha_{i}+k \pmod{s_{i}}$. 
\end{enumerate}
\end{claim}
\begin{proof}
The condition $\widetilde{P}_{i,j}\in\widetilde{C}_{i,k}$ is equivalent to $(\zeta_m^{\mu_i})^{j}d_i^{\mu_i}=(\zeta_m^{\mu_i})^{k}d_i^{\mu_i}$. 
Hence $\widetilde{P}_{i,j}\in\widetilde{C}_{i,k}$ if and only if $j\equiv k\pmod{s_i}$. 

The condition $\widetilde{P}_{i+1,j}\in\widetilde{C}_{i,k}$ is equivalent to $(\zeta_m^{\mu_{i}})^{j}d_{i+1}^{\mu_i}=(\zeta_m^{\mu_i})^{\alpha_{i}+k}d_{i+1}^{\mu_i}$. 
Thus $\widetilde{P}_{i+1,j}\in\widetilde{C}_{i,k}$ if and only if $j\equiv \alpha_{i}+k\pmod{s_i}$
\end{proof}

By Claim~\ref{claim: 1} (i), we have 
\[ 
\{ v_{\widetilde{C}_{i,k}} \mid \widetilde{P}_{i,j}\in\widetilde{C}_{i,k} \}= \{ v_{\widetilde{C}_{i,k}} \mid k=j+c_is_i \mbox{ for some $c_i\in\ZZ$} \}\subset V_{2i-1}^{\gamma}(v_{\widetilde{P}_{1,0}}) 
\]
for each $v_{\widetilde{P}_{i,j}}\in V_{2i-2}^{\gamma}(v_{\widetilde{P}_{1,0}})$. 
By Claim~\ref{claim: 1} (ii), hence, we obtain the following equation;
\begin{align*}
V_{2i}^{\gamma}(v_{\widetilde{P}_{1,0}})&=\bigcup_{v_{\widetilde{P}_{i,j}}\in V_{2i-2}^\gamma(v_{\widetilde{P}_{1,0}})}\left\{ v_{\widetilde{P}_{i+1,j'}} \,\middle|\, j'=j+\alpha_i+b_is_i \mbox{ for some $b_i\in\ZZ$} \right\} \\
&=\left\{ v_{\widetilde{P}_{i+1,j}} \ \middle|\ j=\sum_{i'=1}^{i}\alpha_{i'}+\sum_{i'=1}^{i}b_{i'}s_{i'} \mbox{ for some $b_1,\dots,b_{i}\in\ZZ$} \right\}
\end{align*}
Since $s$ is the greatest common divisor of $s_1,\dots,s_n$, we obtain the assertion. 
\end{proof}

\begin{rem}\label{rem: meridian 2}
Action (\ref{eq: action}) in the proof of Theorem~\ref{thm: gap class} coincides with the monodromy action on $\widetilde{U}''\setminus(\phi'')^{-1}(\mcB)\cong\widetilde{U}'\setminus(\phi')^{-1}(\mcB)$ of a meridian $[m_1]\in\pi_1(\PP^2\setminus\mcB)$ at a point $Q_1\in B_1\setminus\Sing(\mcB)$. 
Indeed, the path 
\[ [0,1]\ni t\mapsto\left(\epsilon\exp(2\pi\sqrt{-1}\,t),\, 0,\, \epsilon^{1/m}\exp(2\pi\sqrt{-1}\,t/m)\right)\in \widetilde{U}''\setminus(\phi'')^{-1}(\mcB) \] 
from $(\epsilon,0,\epsilon^{1/m})$ to $(\epsilon,0,\zeta_m\epsilon^{1/m})$ is a lift of the path 
\[\delta_{\epsilon}:[0,1]\ni t\mapsto (\epsilon\exp(2\pi\sqrt{-1}\,t),0)\in U'\setminus\mcB, \]
where $(x,y)$ is a system of local coordinates of $\PP^2$ at $Q_1$ so that $F=x$ at $Q_1$, and $\zeta_m:=\exp(2\pi\sqrt{-1}/m)$. 
\end{rem}

\section{Artal arrangements of degree $b$}\label{sec: artal arrangement}

In \cite{artal}, Artal studied plane curves $\mcC=E+L_1+L_2+L_3$, where $E$ is a smooth cubic, and $L_i$ ($i=1,2,3$) are non-concurrent inflectional tangents of $E$. 
He proved that a pair $(\mcC_1,\mcC_2)$ of such curves $\mcC_1,\mcC_2$ is a Zariski pair if the three tangent points of $\mcC_1$ are collinear, and those of $\mcC_2$ are not collinear. 
Tokunaga proved the same result by a different way in \cite{tokunaga}. 
In \cite{ban-ben-shi-tok}, such plane curves are called Artal arrangements.  
In \cite{shirane2}, the author defined an Artal arrangement of degree $b\geq 3$ as a plane curve consisting of one smooth curve of degree $b$ and non-concurrent three of its total inflectional tangents. 
In \cite{shirane2}, he partially distinguished the embedded topology of Artal arrangements. 
In this section, we define Artal arrangements of type $(\parti_1,\parti_2,\parti_3)$ for three partitions $\parti_i$ of an integer $d\geq3$, which is a generalization of Artal arrangements defined in \cite{ban-ben-shi-tok} and \cite{shirane2}.

\begin{defin}\label{def: artal arrangement}
Let $B\subset\PP^2$ be a smooth curve of degree $d\geq 3$. 
\begin{enumerate}
\item For a partition $\parti=(e_1,\dots,e_n)$ of $d$, we call a line $L\subset\PP^2$ a \textit{tangent of type $\parti$} of $B$ if $L$ intersects with $B$ at just $n$ points $P_1,\dots,P_n$ with multiplicity $e_1,\dots,e_n$, respectively. 
\item Let $\parti_1,\parti_2$ and $\parti_3$ be three partitions of $d$, and 
assume that there is a tangent $L_i$ of type $\parti_i$ of $B$ for each $i=1,2,3$. 
We call the plane curve $B+L_1+L_2+L_3$ an \textit{Artal arrangement of type $(\parti_1,\parti_2,\parti_3)$} if $L_1\cap L_2\cap L_3=\emptyset$ and $B\cap L_i\cap L_j=\emptyset$ for any $i\ne j$. 
\item For an Artal arrangement $\mcA:=B+L_1+L_2+L_3$, let $\phi_\mcA:X_\mcA\to\PP^2$ be the $\ZZ_d$-cover of type $B$. 
We call $\phi_\mcA$ the \textit{cyclic cover of the Artal arrangement $\mcA$}. 
We fix the surjection $\rho_\mcA:\pi_1(\PP^2\setminus B)\twoheadrightarrow\ZZ_d$ defined by $[m_B]\mapsto [1]$, where $m_B$ is a meridian of $B$ at a point of $B$. 
Furthermore, we call the splitting graph $\mcS_{\phi_{\mcA},L_1+L_2+L_3}$ the \textit{splitting graph of $\mcA$}, and denote it by $\mcS_\mcA$. 
\end{enumerate}
\end{defin}

\begin{rem}\label{rem: partition}
For an element $\sigma$ of the symmetric group $\gtS_3$ of three letters, two Artal arrangements $\mcA_1$ and $\mcA_2$ of type $(\parti_1,\parti_2,\parti_3)$ and $(\parti_{\sigma(1)},\parti_{\sigma(2)},\parti_{\sigma(3)})$, respectively, have the same combinatorics. 
To avoid confusion, we introduce an order on the set of partitions as follows: 

Let $\parti_i=(e_{i,1},\dots,e_{i,n_i})$ ($i=1,2$) be two partitions of $d$ with $1\leq e_{i,j}\leq e_{i,j'}$ for $j< j'$. 
Assume that $\parti_1\ne\parti_2$, and 
put $j_0:=\min\{j\mid e_{1,j}\ne e_{2,j}\}$. 
We write $\parti_1\prec\parti_2$ if $e_{1,j_0}<e_{2,j_0}$. 
We assume that any triple $(\parti_1,\parti_2,\parti_3)$ satisfies $\parti_1\preceq\parti_2\preceq\parti_3$. 
\end{rem}

Let $\parti_i:=(e_{i,1},\dots,e_{i,n_i})$ be a partition of $d\geq 3$ for each $i=1,2,3$, and put $\Parti:=(\parti_1,\parti_2,\parti_3)$. 
Let $\mcF_{\Parti}\subset\PP_\ast\HH^0(\PP^2,\mcO_{\PP^2}(d+3))$ be the family of Artal arrangements of type $\Parti$. 
Here $\PP_\ast\HH^0(\PP^2,\mcO_{\PP^2}(d+3))$ is the projective space of one-dimensional subspaces of the vector space $\HH^0(\PP^2,\mcO_{\PP^2}(d+3))$, 
which parameterizes all plane curves of degree $d+3$. 
Let $s_i$ be the greatest common divisor $\GCD(e_{i,1},\dots,e_{i,n_i})$ for each $i=1,2,3$, 
and put $s:=\GCD(s_1,s_2,s_3)$. 
Let $\mcA:=B+\mcL$ be an Artal arrangement of type $\Parti$, where $\mcL:=L_1+L_2+L_3$. 
Let $\mcG_\mcA$ denote the incidence graph $\mcG_{\phi_\mcA,\mcL}$ of $\mcL$ with respect to $\phi_{\mcA}:X_\mcA\to\PP^2$. 
Let $\gamma^+_\mcA$ be the following cycle on $\mcG_{\mcA}$: 
\begin{align*} 
\gamma^+_\mcA&:=(v_{P_1},v_{L_1},v_{P_2},v_{L_2},v_{P_3},v_{L_3},v_{P_1}), 
\end{align*}
where $P_1$, $P_2$ and $P_3$ are the intersections $L_3\cap L_1$, $L_1\cap L_2$ and $L_2\cap L_3$, respectively.

To compute net voltage classes of closed walks on $\mcG_{\mcA}$ for $\phi_\mcA$, it is enough to compute the net voltage classes of $\gamma^+_\mcA$ and its inverse walk $\gamma^-_\mcA:=(\gamma^+_\mcA)^{-1}$ by Lemma~\ref{lem: gap class of product} and \ref{lem: gap of shifted walk}. 
Note that the splitting number of $L_i$ for $\phi_\mcA$ is equal to $s_i$ by \cite[Theorem~2.7]{shirane}. 
By Theorem~\ref{thm: gap class}, the net voltage class of $\gamma_{\mcA}^+$ for $\phi_{\mcA}$ forms into 
\[ \NV_{\phi_\mcA}(\gamma^+_\mcA)=
[\beta]+s\ZZ_d \]
for some integer $\beta$ with $0\leq\beta<s$. 
By Lemma~\ref{lem: inverse walk}, we obtain $\NV_{\phi_\mcA}(\gamma_\mcA^-)=[-\beta]+s\ZZ_d$. 
For $0\leq\alpha\leq\lfloor s/2\rfloor$, let $\mcF_{\Parti}^\alpha\subset\mcF_{\Parti}$ be the set of Artal arrangements $\mcA$ of type $\Parti$ satisfying 
\[\{\NV_{\phi_\mcA}(\gamma_\mcA^+),\NV_{\phi_\mcA}(\gamma_\mcA^-)\}=\{[\alpha]+s\ZZ_d,[-\alpha]+s\ZZ_d\}, \]
where $\lfloor s/2\rfloor$ is the integer part of $s/2$. 
The family $\mcF_\Parti$ is decomposed into the following disjoint union: 
\[ \mcF_\Parti=\coprod_{\alpha=0}^{\lfloor s/2\rfloor}\mcF_\Parti^\alpha. \] 

\begin{thm}\label{thm: artal arrangement}
Let $\parti_i=(e_{i,1},\dots,e_{i,n_i})$ be three partition of $d\geq 3$ for $i=1,2,3$, and put $\Parti:=(\parti_1,\parti_2,\parti_3)$, $s_i:=\GCD(e_{i,1},\dots,e_{i,n_i})$ for $i=1,2,3$ and $s:=\GCD(s_1,s_2,s_3)$. 
Then, two Artal arrangements $\mcA_1,\mcA_2\in\mcF_{\Parti}$ have the same embedded topology if and only if the splitting graphs $\mcS_{\mcA_1}$ and $\mcS_{\mcA_2}$ are equivalent, $\mcS_{\mcA_1}\sim\mcS_{\mcA_2}$. 
Moreover, the followings hold: 
\begin{enumerate}
\item In the case where $\parti_i\ne\parti_j$ for any $i\ne j$, $\mcF_\Parti^\alpha$ consists of two connected components $\mcF_\Parti^{\alpha+}$ and $\mcF_\Parti^{\alpha-}$ if $0<\alpha<s/2$, and $\mcF_\Parti^\alpha$ is connected otherwise, i.e., either $\alpha=0$ or $\alpha=s/2$ if $s$ is even. 
\item In the case where $\parti_i=\parti_j$ for some $i\ne j$, $\mcF_\Parti^\alpha$ is connected for each $0\leq \alpha\leq \lfloor s/2\rfloor$. 
\item 
Let $\mcA$ be an Artal arrangement of $\mcF_{\Parti}^\alpha$ $(0\leq\alpha\leq\lfloor s/2\rfloor)$, and let $
\bar{h}:\PP^2\to\PP^2$ be the base change given by the complex conjugate homomorphism $\CC\to\CC$ $(z\mapsto\bar{z})$, which is a homeomorphism. 
Then $\bar{h}(\mcA)\in\mcF_{\Parti}^{\alpha}$. 
Moreover, in the case where $\parti_i\ne\parti_j$ for $i\ne j$ and $0<\alpha< s/2$, $\bar{h}(\mcA)\in\mcF_\Parti^{\alpha-}$ if $\mcA\in\mcF_\Parti^{\alpha+}$. 
\item For two Artal arrangements $\mcA_1, \mcA_2\in\mcF_\Parti$, there is a homeomorphism $h:\PP^2\to\PP^2$ with $h(\mcA_1)=\mcA_2$ if and only if $\mcA_1,\mcA_2\in\mcF_\Parti^{\alpha}$ for some $0\leq\alpha\leq\lfloor s/2\rfloor$. 
\end{enumerate}
\end{thm}

In order to prove Theorem~\ref{thm: artal arrangement}, we prove four lemmas. 
We first seek a simple defining equation of an Artal arrangement, up to projective transforms of $\PP^2$. 
Let $L_x, L_y$ and $L_z$ be the lines defined by $x=0$, $y=0$ and $z=0$, respectively, and 
put $\mcL_{xyz}:=L_x+L_y+L_z$. 

\begin{lem}\label{lem: defining eq.}
Let $\mcA:=B+\mcL$ be an Artal arrangement of type $\Parti=(\parti_1,\parti_2,\parti_3)$, 
where $\mcL:=L_1+L_2+L_3$. 
Put $\mu_{i,j}:=e_{i,j}/s_i$. 
Then, after a certain projective transform of $\PP^2$, $\mcL$ satisfies $L_1=L_x$, $L_2=L_y$ and $L_3=L_z$, and 
$B$ is defined by $F_\Parti(\beta,\{c_{i,j}\},g_0)=0$ with 
{\footnotesize
\begin{align*}
&F_{\Parti}(\beta,\{c_{i,j}\},g_0) \\ :=&\prod_{j=1}^{n_1}(y+c_{1,j}z)^{e_{i,j}}+\prod_{j=1}^{n_2}(z+c_{2,j}x)^{e_{2,j}}+\prod_{j=1}^{n_3}(x+c_{3,j}y)^{e_{3,j}}-x^d-y^d-z^d+xyzg_0, \notag
\end{align*}
}
where $\beta$ is an integer with $0\leq\beta<s$, $g_0$ is a homogeneous polynomial of degree $d-3$ in $x,y,z$, and
$c_{i,j}$ are complex numbers satisfying the following conditions; 
\begin{align}\label{eq: condition of c} 
\prod_{j=1}^{n_i}c_{i,j}^{\mu_{i,j}}=1 \ \ (i=1,2), & &
\prod_{j=1}^{n_3}c_{3,j}^{\mu_{3,j}}=\zeta_d^{\beta\mu_3}, &&
c_{i,j}\ne c_{i,j'} \ \mbox{ if $j\ne j'$}, 
\end{align}
where $\mu_i:=d/s_i=\sum_{j=1}^{n_i}\mu_{i,j}$ for $i=1,2,3$. 
\end{lem}

\begin{proof}
It is clear that $\mcL$ satisfies $L_1=L_x$, $L_2=L_y$ and $L_3=L_z$ after a projective transform. 
Let $Q_{i,j}$ be the intersection point of $B$ and $L_i$ with $\I_{Q_{i,j}}(B,L_i)=e_{i,j}$ for each $i=1,2,3$ and $j=1,\dots,n_i$, and let $a_{i,j}$ be the complex number such that 
\begin{align*} 
Q_{1,j}&=(0{:}-a_{1,j}{:}1), &
Q_{2,j}&=(1{:}0{:}-a_{2,j}), &
Q_{3,j}&=(-a_{3,j}{:}1{:}0).
\end{align*}
Let $F=0$ be a defining equation of $B$. 
Since $B\cap L_i\cap L_j=\emptyset$ ($i\ne j$), we may assume that the coefficients of $x^d, y^d$ and $z^d$ in $F$ are equal to $1$. 
Since $\I_{Q_{i,j}}(B,L_i)=e_{i,j}$, a homogenous polynomial $F$ forms into 
\begin{align*}
F
&=\prod_{j=1}^{n_1}(y+a_{1,j}z)^{e_{1,j}}+g_1+xyzg_0 
\\
&=\prod_{j=1}^{n_2}(z+a_{2,j}x)^{e_{2,j}}+g_2+xyzg_0 
\\
&=\prod_{j=1}^{n_3}(x+a_{3,j}y)^{e_{3,j}}+g_3+xyzg_0, 
\end{align*}
where $g_1,g_2$ and $g_3$ are the sum of terms of $F$ which are not divisible by $xyz$, but by $x,y$ and $z$, respectively. 
Note that $\prod_{j=1}^{n_i}a_{i,j}^{e_{i,j}}=1$. 
Then the coefficient of $y^k z^{d-k}$ ($0<k<d$) in $F$ is equal to the one in $\prod_{j=1}^{n_1}(y+a_{i,j}z)^{e_{i,j}}$. 
Similarly, the coefficients of $x^k y^{d-k}$ and $x^k z^{d-k}$ are determined by the above equation. 
Hence we have 
{\small \begin{equation*}
F:=\prod_{j=1}^{n_1}(y+a_{1,j}z)^{e_{i,j}}+\prod_{j=1}^{n_2}(z+a_{2,j}x)^{e_{2,j}}+\prod_{j=1}^{n_3}(x+a_{3,j}y)^{e_{3,j}}-x^d-y^d-z^d+xyzg_0. 
\end{equation*} }
Since $\prod_{j=1}^{n_i}a_{i,j}^{e_{i,j}}=(\prod_{j}^{n_i}a_{i,j}^{\mu_{i,j}})^{s_i}=1$, 
we have 
\[ \prod_{j=1}^{n_i}a_{i,j}^{\mu_{i,j}}=(\zeta_d^{\mu_i})^{\beta_i} \]
for some integer $0\leq\beta_i<s_i$ ($i=1,2,3$). 
After the projective transform defined by $x\mapsto\zeta_d^{-\beta_1-\beta_2}x$, $y\mapsto y$ and $z\mapsto\zeta_d^{-\beta_1}z$, we obtain 
{\small
\[ F=\prod_{j=1}^{n_1}(y+b_{1,j}z)^{e_{1,j}}+\prod_{j=1}^{n_2}(z+b_{2,j}x)^{e_{2,j}}+\prod_{j=1}^{n_3}(x+b_{3,j}y)^{e_{3,j}}-x^d-y^d-z^d+xyzg_0, \]
}
where $b_{1,j}:=\zeta_d^{-\beta_1}a_{1,j}$, $b_{2,j}:=\zeta_d^{-\beta_2}a_{2,j}$ and $b_{3,j}:=\zeta_d^{\beta_1+\beta_2}a_{3,j}$. 
We have $\prod_{j=1}^{n_i}b_{i,j}=1$ for $i=1,2$, and $\prod_{j=1}^{n_3}b_{3,j}=(\zeta_d^{\mu_3})^{\beta'}$ for some integer $0\leq\beta'<s_3$. 
Suppose that $ks\leq\beta'<(k+1)s$. 
Let $b_1,b_2$ and $b_3$ be three integers so that $b_1s_1+b_2s_2+b_3s_3=ks$. 
By the projective transform given by $x\mapsto\zeta_d^{b_1s_1+b_2s_2}x$, $y\mapsto y$ and $z\mapsto\zeta_d^{b_1s_1}z$, 
we obtain $F=F_\Parti(\beta,\{c_{i,j}\},g_0)$, where $c_{i,j}:=\zeta_{b_is_i}b_{i,j}$ for $i=1,2$, $c_{3,j}:=\zeta_d^{-b_1s_1-b_2s_2}b_{3,j}$, and $\beta=\beta'-ks$. 
\end{proof}

Next, we prove that the curve defined by $F_\Parti(\beta,\{c_{i,j}\},g_0)=0$ is smooth for a general polynomial $g_0$. 

\begin{lem}\label{lem: smooth}
Fix an integer $\beta$ with $0\leq\beta<s$, and let $c_{i,j}$ be complex numbers satisfying (\ref{eq: condition of c}). 
Then the equation $F_\Parti(\beta,\{c_{i,j}\},g_0)=0$ defines a smooth curve $B$ on $\PP^2$ for a general homogeneous polynomial $g_0$. 
\end{lem}
\begin{proof}

We consider the linear system $\Lambda$ consisting of curves defined by $aF_\Parti(\beta,\{c_{i,j}\},g_0)+bxyzg=0$ for $(a{:}b)\in\PP^1$ and $g\in\HH^0(\PP^2,\mcO_{\PP^2}(d-3))\setminus\{0\}$. 
Since the base points of $\Lambda$ are $Q_{1,j}=(0{:}-c_{1,j}{:}1), Q_{2,j}=(1{:}0{:}-c_{2,j}), Q_{3,j}=(-c_{3,j}{:}1{:}0)$, 
a general member of $\Lambda$ is smooth except for the base points $Q_{i,j}$ by Bertini's theorem (see \cite{kleiman}). 
Since $xyzg=0$ defines a curve smooth at all base points $Q_{i,j}$ if $g(Q_{i,j})\ne 0$ for any $i,j$, a general member of $\Lambda$ is smooth at $Q_{i,j}$. 
Therefore $F_\Parti(\beta,\{c_{i,j}\},g_0)=0$ defines a smooth curve $B$ on $\PP^2$ for a general homogeneous polynomial~$g_0$. 
\end{proof}

\begin{rem}
If $F_\Parti(\beta,\{c_{i,j}\},g_0)=0$ defines a smooth curve $B$, then $\mcA=B+\mcL_{xyz}$ is an Artal arrangement of type $\Parti=(\parti_1,\parti_2,\parti_3)$. 
In this case, we put $\gamma_{\mcA}^+:=(v_{P_1},v_{L_x},v_{P_2},v_{L_y},v_{P_3},v_{L_z},v_{P_1})$. 
\end{rem}

\begin{lem}\label{lem: alpha}
Fix an integer $\beta$ with $0\leq\beta<s$ 
and $c_{i,j}$ satisfying (\ref{eq: condition of c}). 
Put $\alpha:=\beta$ if $\beta\leq\lfloor s/2\rfloor$, and $\alpha:=s-\beta$ if $\beta>\lfloor s/2\rfloor$. 
Assume that $F_\Parti(\beta,\{c_{i,j}\},g_0)=0$ defines a smooth curve $B\subset\PP^2$.
Then 
the equation $\NV_{\phi_\mcA}(\gamma_\mcA^+)=[\beta]+s\ZZ_d$ holds. 
In particular, 
the Artal arrangement $\mcA:=B+\mcL_{xyz}$ is a member of $\mcF_{\Parti}^\alpha$. 
\end{lem}
\begin{proof}
Let $L$ be the line defined by $x+y+z=0$, which intersects transversally with $\mcL_{xyz}$, and is not a component of $B$. 
We compute the net voltage class $\NV_{\phi_\mcA}(\gamma_\mcA^+)$ 
by using Theorem~\ref{thm: gap class}. 
Let $P_1, P_2$ and $P_3$ be the singular points $(0{:}1{:}0)$, $(0{:}0{:}1)$ and $(1{:}0{:}0)$ of $\mcL_{xyz}$. 
The equation $F_\Parti(\beta,\{c_{i,j}\},g_0)/(x+y+z)^d=0$ is a defining equation of $B$ on $U:=\PP^2\setminus L$. 
We seek $h_i$ in (\ref{eq: g and h}) for $L_i$ ($i=x,y,z$), and compute $h_i(P_j)$ for $P_j\in L_i$. 
Since $L_x$ is defined by $x=0$ and $s_1=\GCD(e_{1,1},\dots,e_{1,n_1})$, 
we obtain 
\[ h_x=\frac{\prod_{j=1}^{n_1}(y+c_{1,j}z)^{\mu_{1,j}}}{(x+y+z)^{\mu_1}}. \]
Thus we have $h_x(P_1)=h_x(P_2)=1$. 
Similarly, we have $h_y(P_2)=h_y(P_3)=1$, $h_z(P_3)=1$ and $h_z(P_1)=(\zeta_d^{\mu_3})^\beta$. 
Hence, by Theorem~\ref{thm: gap class}, we obtain $\NV_{\phi_\mcA}(\gamma_\mcA^+)=[\beta]+s\ZZ_d$. 
Moreover, we obtain $\NV_{\phi_\mcA}(\gamma_\mcA^-)=[-\beta]+s\ZZ_d$ by Lemma~\ref{lem: inverse walk}. 
Therefore we have 
\[ \{\NV_{\phi_\mcA}(\gamma_\mcA^+),\NV_{\phi_\mcA}(\gamma_\mcA^-)\}=\{[\beta]+s\ZZ,[-\beta]+s\ZZ_d\}=\{[\alpha]+s\ZZ_d,[-\alpha]+s\ZZ_d\}. \qedhere\]
\end{proof}

Next we prove that $B+\mcL_{xyz}$ and $B'+\mcL_{xyz}$ are two members of a connected component of $\mcF_\Parti$ 
if $B$ and $B'$ are smooth curves defined by $F_\Parti(\beta,\{c_{i,j}\},g_0)=0$ and $F_\Parti(\beta,\{c_{i,j}'\},g_0')=0$, respectively. 

\begin{lem}\label{lem: connected component}
Fix an integer $0\leq \beta<s$. 
If $B$ and $B'$ are smooth curves defined by $F_\Parti(\beta,\{c_{i,j}\},g_0)=0$ and $F_\Parti(\beta,\{c'_{i,j}\},g_0')=0$, respectively, 
then the Artal arrangements $\mcA:=B+\mcL_{xyz}$ and $\mcA':=B'+\mcL_{xyz}$ are members of a connected component of $\mcF_\Parti$. 
In particular, there is a homeomorphism $h:(\PP^2,\mcA)\to(\PP^2,\mcA')$ such that $h_\ast([m_B])=[m_{B'}]$ for meridians $m_B$ and $m_{B'}$ of $B$ and $B'$, respectively. 
\end{lem}
\begin{proof}
Let $U_\sm \subset(\CC^\times)^{n_1+n_2+n_3}\times\HH^0(\PP^2,\mcO_{\PP^2}(d-3))$ be the following subset: 
\begin{align*} 
U_\sm :=
\left\{
\left((c_{i,j}), g_0\right) \,\middle|\,
\begin{array}{l} 
\mbox{$(c_{i,j})\in\CC^{n_1+n_2+n_3}$ satisfies equation (\ref{eq: condition of c}), and} \\ \mbox{$F_\Parti(\beta,\{c_{i,j}\},g_0)=0$ defines a smooth curve}
\end{array}
\right\}. 
\end{align*}
It is enough to prove that $U_\sm$ is connected. 
Let $V'_i$ be the following subset of $(\CC^\times)^{n_i}$ for $i=1,2,3$: 
\[ V_i':=\left\{(c_{i,1},\dots,c_{i,n_i})\in(\CC^\times)^{n_i}\ \middle| \ \prod_{j=1}^{n_i}c_{i,j}^{\mu_{i,j}}=(\zeta_d^{\mu_i})^{\beta_i}\right\}, \]
where $\beta_1=\beta_2=0$ and $\beta_3=\beta$. 
Consider the projection $\pr_i:V_i'\to(\CC^\times)^{n_i-1}$ defined by $(c_{i,1},\dots,c_{i,n_i-1},c_{i,n_i})\mapsto(c_{i,1},\dots,c_{i,n_i-1})$. 
Since $c_{i,n_i}^{\mu_{i,n_i}}=(\zeta_d^{\mu_i})^{\beta_i}\prod_{j=1}^{n_i-1}c_{i,j}^{-\mu_{i,j}}$, the preimage of $\pr_i$ at a point of $(\CC^\times)^{n_i-1}$ consists of just $\mu_{i,n_i}$ points. 
Moreover, since $\GCD(\mu_{i,1},\dots,\mu_{i,n_i})=1$, 
$\pr_i$ is the unramified covering induced by the surjection $\pi_1((\CC^\times)^{n_i-1})\twoheadrightarrow\ZZ_{\mu_{i,n_i}}$ which maps any meridian of $\{c_{i,j}=0\}$ to $[-\mu_{i,j}]\in\ZZ_{\mu_{i,n_i}}$ for $j=1,\dots,n_i-1$. 
Hence $V_i'$ is smooth and irreducible. 
Let $U_i'$ be the following subset of $(\CC^\times)^{n_i}$: 
\[ U_i':=\{(c_{i,1},\dots,c_{i,n_i})\in(\CC^\times)^{n_i}\mid c_{i,j}\ne c_{i,j'} \ (j\ne j')\}. \]
Since $U_i'$ is a Zariski open subset of $(\CC^\times)^{n_i}$, $V_i:=U_i'\cap V_i'$ is a Zariski open subset of $V_i'$, hence $V_i$ is connected. 
For any $(c_{i,j})\in V_1\times V_2\times V_3$, $F_\Parti(\beta,\{c_{i,j}\},g_0)=0$ defines a smooth curve for a general $g_0$ by Lemma~\ref{lem: smooth}. 
Therefore, $U_\sm$ is a non-empty open subset of $V_1\times V_2\times V_3\times\HH^0(\PP^2,\mcO_{\PP^2}(d-3))$, 
and $U_\sm$ is connected. 
\end{proof}

We prove Theorem~\ref{thm: artal arrangement} by using Lemmas~\ref{lem: defining eq.}, \ref{lem: smooth}, \ref{lem: alpha} and \ref{lem: connected component}.

\begin{proof}[Proof of Theorem~\ref{thm: artal arrangement}]
Suppose that the splitting graphs $\mcS_{\mcA_1}$ and $\mcS_{\mcA_2}$ are equivalent, i.e., 
there exist a homeomorphism $h':\mcT(\mcA_1)\to\mcT(\mcA_2)$ and an automorphism $\tau:\ZZ_d\to\ZZ_d$ satisfying conditions (i), (ii), (iii) in Definition~\ref{def: equivalence of splitting graph}. 
By the proof of Lemma~\ref{lem: meridian}, we have $\tau([1])=[\pm 1]$ since $\rho_{\mcA_i}([m_B])=[1]$. 
By Corollary~\ref{cor: gap class}, we obtain the following equation:
\begin{equation}\label{eq: gap class}
\{\NV_{\phi_{\mcA_1}}(\gamma_{\mcA_1}^+),\NV_{\phi_{\mcA_1}}(\gamma_{\mcA_1}^-)\}=\{\NV_{\phi_{\mcA_2}}(\gamma_{\mcA_2}^+),\NV_{\phi_{\mcA_2}}(\gamma_{\mcA_2}^-)\}. 
\end{equation}
Hence we obtain $\mcA_1,\mcA_2\in\mcF_\Parti^\alpha$ for some $0\leq\alpha\leq\lfloor s/2\rfloor$. 

Note that, if there exists a homeomorphism $h:(\PP^2,\mcA_1)\to(\PP^2,\mcA_2)$ for Artal arrangements $\mcA_i\in\mcF_\Parti$ ($i=1,2$), 
then $\mcS_{\mcA_1}\sim\mcS_{\mcA_2}$ by Theorem~\ref{thm: splitting graph} 
since $h$ must satisfy $h(B_1)=B_2$ and the isomorphism $h_\ast:\pi_1(\PP^2\setminus B_1)\to\pi_1(\PP^2\setminus B_2)$ is given by $h_\ast([m_{B_1}])=[m_{B_2}]^{\pm 1}$ (cf. Remark~\ref{rem: smooth curve}). 
Furthermore, 
if two Artal arrangements $\mcA_i:=B_i+\mcL_i\in\mcF_\Parti$ is members of the same connected component of $\mcF_\Parti$, then there exists a homeomorphism $h:(\PP^2,\mcA_1)\to(\PP^2,\mcA_2)$ with $h_\ast([m_{B_1}])=[m_{B_2}]$. 
Thus it is enough to prove assertions (i), (ii), (iii) and (iv) in Theorem~\ref{thm: artal arrangement}.

(i) 
Suppose that $\parti_i\ne\parti_j$ for any $i\ne j$. 
Let $\mcA_i:=B_i+\sum_{j=1}^3L_{i,j}$ be an Artal arrangement of type $\Parti$ for each $i=1,2$. 
Assume that there is a homeomorphism $h:(\PP^2,\mcA_1)\to(\PP^2,\mcA_2)$ such that $h_\ast([m_{B_1}])=[m_{B_2}]$ for meridians $m_{B_i}$ of $B_i$. 
In this case, $h$ must satisfy $h(L_{1,j})=L_{2,j}$ for $j=1,2,3$ (see Remark~\ref{rem: partition}), and if an automorphism $\tau:\ZZ_d\to\ZZ_d$ satisfies $\tau\circ\rho_{\mcA_1}=\rho_{\mcA_2}\circ h_\ast$, then $\tau([1])=[1]$. 
Hence we have $\theta_h(\gamma_{\mcA_1}^+)=\gamma_{\mcA_2}^+$ and $\NV_{\phi_{\mcA_1}}(\gamma_{\mcA_1}^+)=\NV_{\phi_{\mcA_2}}(\gamma_{\mcA_2}^+)$. 

For $0<\alpha<s/2$, let $\mcF_\Parti^{\alpha\pm}$ be the subsets consisting of $\mcA\in\mcF_\Parti$ with $\NV_{\phi_{\mcA}}(\gamma_\mcA^+)=[\pm\alpha]+s\ZZ_d$, respectively. 
By Lemmas~\ref{lem: defining eq.}, \ref{lem: alpha} and \ref{lem: connected component}, $\mcF_\Parti^{\alpha\pm}$ are connected. 
Since $[\alpha]+s\ZZ_d\ne[-\alpha]+s\ZZ_d$ for $0<\alpha<s/2$, $\mcF_\Parti^{\alpha+}\cap\mcF_\Parti^{\alpha-}=\emptyset$ by Lemma~\ref{lem: connected component}. 

For $\alpha=0$ or $\alpha=s/2$ if $s$ is even, we have $[\alpha]+s\ZZ_d=[-\alpha]+s\ZZ_d$. 
Hence, by Lemmas~\ref{lem: defining eq.} and \ref{lem: alpha}, 
an Artal arrangement $\mcA\in\mcF_\Parti^\alpha$ is projective equivalent to a curve defined by $F_\Parti(\alpha,\{c_{i,j}\},g_0)=0$. 
Therefore, $\mcF_\Parti^\alpha$ is connected by Lemma~\ref{lem: connected component}. 

(ii) 
Suppose that $\parti_2=\parti_3$. 
Let $B$ be a smooth curve defined by $F_\Parti(\alpha,\{c_{i,j}\},g_0)$, and 
put $\mcA:=B+\mcL_{xyz}$. 
Let $h:\PP^2\to\PP^2$ be the projective transformation defined by $x\mapsto x$, $y\mapsto z$ and $z\mapsto y$, and 
put $\mcA':=h(\mcA)=h(B)+\mcL_{xyz}$. 
Then $\theta_{h}(\gamma_\mcA^\pm)=\gamma_{\mcA'}^\mp$, respectively, up to cyclic permutations of vertices (see Lemma~\ref{lem: gap of shifted walk}). 
Hence $\mcA$ and $\mcA'$ are members of a connected component of $\mcF_\Parti^\alpha$, 
and $\mcF_\Parti^\alpha$ is connected by Lemmas~\ref{lem: defining eq.}, \ref{lem: alpha} and \ref{lem: connected component}. 

(iii) 
Let $\mcA:=B+L_1+L_2+L_3$ be an Artal arrangement in $\mcF_\Parti^{\alpha}$ with $\NV_{\phi_\mcA}(\gamma_{\mcA}^+)=[\alpha]+s\ZZ_d$, and put $\bar{B}:=\bar{h}(B)$. 
Let $m_B$ be a meridian of $B$ at a point $P\in B$. 
By the definition of meridians, $(\bar{h}\circ m_B)^{-1}$ is a meridian of $\bar{B}$ at $\bar{h}(P)\in \bar{B}$. 
Hence we have $\tau([1])=[-1]$ for the automorphism $\tau:\ZZ_d\to\ZZ_d$ such that $\rho_{\bar{h}(\mcA)}\circ \bar{h}_{\ast}=\tau\circ\rho_\mcA$, 
Thus we obtain $\NV_{\phi_{\bar{h}(\mcA)}}(\gamma_{\bar{h}(\mcA)}^+)=[-\alpha]+s\ZZ_d$, and $\bar{h}(\mcA)\in\mcF_\Parti^\alpha$. 
In particular, in the case where $\parti_i\ne\parti_j$ for any $i\ne j$ and $0<\alpha<s/2$, if $\mcA\in\mcF_\Parti^{\alpha+}$, then $\bar{h}(\mcA)\in\mcF_\Parti^{\alpha-}$. 

(iv) The assertion follows from (i), (ii) and (iii). 
\end{proof}

As a corollary of Theorem~\ref{thm: artal arrangement}, we obtain Zariski $k$-plets of Artal arrangements. 

\begin{cor}\label{cor: zariski pair}
Let $\parti_i$ be a partition $(e_{i,1},\dots,e_{i,n_i})$ of $d\geq 3$ for each $i=1,2,3$. 
Let $s$ be the greatest common divisor of $e_{i,j}$, $i=1,2,3$ and $j=1,\dots,n_i$. 
Then there is a Zariski $(\lfloor s/2\rfloor+1)$-plet $(\mcA_0,\dots,\mcA_{\lfloor s/2\rfloor})$ of Artal arrangements $\mcA_i$ of type $\Parti$, i.e., $(\mcA_i,\mcA_j)$ is a Zariski pair for any $0\leq i< j\leq\lfloor s/2\rfloor$. 
\end{cor}

\begin{proof}
Let $\mcA_\alpha$ be a member of $\mcF_\Parti^\alpha$ for each $\alpha=0,\dots,\lfloor s/2\rfloor$. 
Then $(\mcA_0,\dots,\mcA_{\lfloor s/2\rfloor})$ is a Zariski $(\lfloor s/2\rfloor+1)$-plet by Theorem~\ref{thm: artal arrangement}. 
\end{proof}

\noindent
\textbf{Acknowledgement.} 
The author thanks Professor Alex Degtyarev for his useful comment. 
He also thanks Professor Hiro-o Tokunaga and Professor Shinzo Bannai for valuable discussions and comments.

\end{document}